\documentclass[11pt,a4paper]{amsart}

\title[Conformal dimension via subcomplexes]{Conformal dimension via subcomplexes for 
small cancellation and random groups}
\author{John M. Mackay}
\address{School of Mathematics \\
 University of Bristol \\ Bristol, UK.}
\email{john.mackay@bristol.ac.uk}
\date{\today}

\usepackage{amsmath,amsthm,amsfonts,graphicx,amssymb,color}
\usepackage{hyperref}

\numberwithin{equation}{section}
\newtheorem{theorem}[equation]{Theorem}
\newtheorem{proposition}[equation]{Proposition}
\newtheorem{corollary}[equation]{Corollary}
\newtheorem{lemma}[equation]{Lemma}

\newtheorem{definition}[equation]{Definition}
\newtheorem{remark}[equation]{Remark}
\newtheorem{assumption}[equation]{Assumption}

\newtheoremstyle{citing}
  {3pt}
  {3pt}
  {\itshape}
  {}
  {\bfseries}
  {}
  {.5em}
  {\thmnote{#3}}

\theoremstyle{citing}

\DeclareMathOperator{\diam}{diam}

\DeclareMathOperator{\St}{St}

\renewcommand{\mod}{\operatorname{mod}}
\newcommand{\alp}{\alpha}

\newcommand{\gam}{\gamma}
\newcommand{\Gam}{\Gamma}
\newcommand{\eps}{\epsilon}

\newcommand{\lam}{\lambda}

\newcommand{\bdry}{\partial_\infty}

\newcommand{\cV}{\mathcal{V}}

\DeclareMathOperator{\Cdim}{Confdim}

\newcommand{\cR}{\mathcal{R}}
\newcommand{\cS}{\mathcal{S}}

\newcommand{\ra}{\rightarrow}
\newcommand{\R}{\mathbb{R}}
\newcommand{\Sph}{\mathbb{S}}
\newcommand{\N}{\mathbb{N}}
\newcommand{\Z}{\mathbb{Z}}

\newcommand{\ba}{\mathbf{a}}
\newcommand{\bbb}{\mathbf{b}}

\def\XXint#1#2#3{{\setbox0=\hbox{$#1{#2#3}{\int}$}
\vcenter{\hbox{$#2#3$}}\kern-.5\wd0}}

\numberwithin{equation}{section}

\begin{document}

\begin{abstract}
We find new bounds on the conformal dimension of small cancellation groups.
These are used to show that a random few relator group
has conformal dimension $2 + o(1)$ asymptotically almost surely (a.a.s.).
In fact, if the number of relators grows like $l^K$ in the length $l$ of the relators,
then a.a.s.\ such a random group has conformal dimension $2+K+ o(1)$.
In Gromov's density model, a random group at 
density $d<\frac{1}{8}$ a.a.s.\ has conformal dimension
$\asymp dl / |\log d|$.

The upper bound for $C'(\frac{1}{8})$ groups has two main ingredients:
$\ell_p$-cohomology (following Bourdon--Kleiner), 
and walls in the Cayley complex (building on Wise and Ollivier--Wise).
To find lower bounds we refine the methods of \cite{Mac-12-random-cdim} to
create larger `round trees' in the Cayley complex of such groups.

As a corollary, in the density model at $d<\frac{1}{8}$, the density $d$ 
is determined, up to a power, by the conformal dimension of the boundary
and the Euler characteristic of the group.
\end{abstract}

\maketitle

\section{Introduction}\label{sec-intro}

\subsection{Overview}
The large scale geometry of a hyperbolic group $G$ is captured by the metric properties
of its boundary at infinity $\bdry G$.
An important invariant of the boundary is its conformal dimension,
introduced by Pansu,
which gives a quasi-isometry invariant of the group.
In this paper we are concerned with the challenge of estimating this analytic 
invariant when we are only given the algebraic information of a group presentation.

One main motivation for this problem is the study of
typical properties of finitely presented groups. 
In various models, random groups are usually hyperbolic, with boundary at infinity
$\bdry G$ homeomorphic to the Menger sponge.
However, we will see that good bounds on conformal dimension 
can be used to illuminate the rich variety of random groups that
arise. 

The conformal dimension $\Cdim(\bdry G)$ of $\bdry G$ is the infimal Hausdorff dimension among
a family of metrics on $\bdry G$ (see Section~\ref{sec-bourdon-kleiner}).
This definition is challenging to work with directly, 
so instead we use tools which allow us to bound the conformal dimension
by building certain subcomplexes in the Cayley complex of $G$.

Bourdon and Kleiner~\cite{BK-12-lp-bdry} recently found good
upper bounds for the conformal dimension of Fuchsian buildings
(and other similar polygonal $2$-complexes)
by splitting the building up into many pieces using 
quasi-convex embedded trees, which we call `branching walls' 
(see Figure~\ref{fig-Vtree} later).
These decompositions are then used to build a rich collection of 
non-trivial $\ell_p$-co\-ho\-mol\-ogy classes, 
which in turn give an upper bound on the conformal dimension 
(see Section~\ref{sec-bourdon-kleiner}).

Our first main result shows how their ideas
can be used to find new upper bounds for the conformal
dimension of any one-ended group with a $C'(\frac{1}{8})$ 
small cancellation presentation.
\begin{theorem}\label{thm-upper-bound}
	Suppose $G=\langle \cS | \cR \rangle$ is a one-ended $C'(\lambda)$ small cancellation group,
	for some $\lambda \leq \frac{1}{8}$,
	with all relations having length $\leq M$, 
	and each generator appearing in at most $k \leq |\cR|M$ locations in the relations.
	Then 
	\[ 
		\Cdim(\bdry G) \leq 1+\frac{\log(k-1)}{\log(\lfloor 1/8\lam \rfloor +1)}
		\leq 1+\frac{\log (|\cR|M)}{\log(\lfloor 1/8\lam \rfloor +1)}.
	\]
\end{theorem}

Once one finds the branching walls, all their machinery goes through to 
the case of Cayley complexes of small cancellation groups.
However, constructing branching walls, which was fairly straightforward in the
case of Fuchsian buildings (or other CAT($-1$) $2$-complexes), 
is now significantly more involved.
We take inspiration from other recent work, unconnected to conformal dimension:
to show that $C'(\frac{1}{6})$ groups and certain random groups have the Haagerup property,
Wise \cite{Wise-04-cubulating-sc-groups} and Ollivier--Wise \cite{Oll-Wis-11-rand-grp-T}
built embedded trees (walls) that split the Cayley complex into two pieces
(see also~\cite{Mac-Prz-14-walls}).
We use similar ideas and tools to build branching walls that split
the Cayley complex into many pieces (see Sections~\ref{sec-wise-walls}--\ref{sec-new-upper}).

Lower bounds for conformal dimension for small cancellation groups
were found in \cite{Mac-12-random-cdim} by building subcomplexes 
quasi-isometric to one of Gromov's `round trees' (see Section~\ref{sec-new-lower-sc}).
In this paper we give new constructions of round trees which give better bounds
for random groups, and which apply in more situations than before, as discussed 
below.
One point to emphasise is that since our lower and upper bounds 
nearly match for random groups, 
we see the strengths of these methods for bounding conformal dimension.

\subsection{Random groups}
A major motivation for looking at the conformal dimension of small cancellation groups is 
the study of \emph{random groups}.
\begin{definition}\label{def:rand-group}
	Consider an integer $m \geq 2$, a function $n = n(l):\N\ra\N$,
	and a property $P$ of a group (presentation).
	Given $l \in \N$, let 
	$G = \langle s_1, \ldots, s_m | r_1, \ldots, r_n \rangle$ 
	be a group presentation where each $r_i$ is chosen independently and uniformly
	from the set of all cyclically reduced words of length $l$ (or $\leq l$)
	in $\langle s_1, \ldots, s_m\rangle$.
	
	A random $m$ generator, $n=n(l)$ relator group has property $P$ \emph{asymptotically
	almost surely (a.a.s.)} if the probability such $G$ has $P$ goes to $1$ as $l \ra \infty$.	
\end{definition}
In the case that $n \in \N$ is a constant, this is the \emph{few relator model} of a random group;
here we let the words have lengths $\leq l$.
There are roughly $(2m-1)^l$ cyclically reduced words of length $l$, so it is natural to let $n$
grow as $l$ grows.
If we fix $d \in (0,1)$, let $n = (2m-1)^{ld}$, and consider words of length exactly $l$,
this is Gromov's \emph{density model} of a random group~\cite[9.B]{Gro-91-asymp-inv}.
(Of course, $(2m-1)^{ld}$ need not be an integer; one should read this as $n=\lfloor (2m-1)^{ld} \rfloor$.)

A random group at density $d > \frac12$ is trivial or $\Z/2\Z$ a.a.s.
On the other hand, if $0<d<\frac{1}{2}$, a random group $G$ at density $d$ is, a.a.s, infinite, hyperbolic,
and has boundary at infinity homeomorphic to the Menger sponge~\cite{DGP-10-density-menger}.
Likewise, a.a.s.\ a random few relator group has the same properties~\cite{Cha-95-rand-grps}.
Given that the topology of the boundary cannot distinguish quasi-isometry classes amongst such groups,
it is natural to consider the conformal dimension of such boundaries (e.g.\ \cite[IV.b]{Oll-05-rand-grp-survey}).

In previous work \cite[Theorems 1.3 and 1.4]{Mac-12-random-cdim}, 
we showed that there is a constant $C>1$ so that
a.a.s.\ a random $m$-generated, few relator group $G$ has 
\[
	1+\frac{1}{C} \leq \Cdim(\bdry G) \leq C \log(2m-1) \cdot \frac{l}{\log l},
\]
while a random group at density $0<d<\frac{1}{16}$ has 
\[
	\frac{d}{C} \cdot \frac{l}{\log l} \leq \Cdim(\bdry G) \leq \frac{C \log(2m-1)}{|\log d\, |} \cdot l.
\]
As a corollary, at density $d<\frac{1}{16}$ we see infinitely many different quasi-isometry classes as $l \ra \infty$.

In this paper, we show the following bounds in the density model.
(The notation $A \asymp_C B$ signifies that $B/C \leq A \leq CB$.)
\begin{theorem}\label{thm-main-density}
	There exists a constant $C>1$ so that if $G$ is a random group at density $d < \frac{1}{8}$, then a.a.s.
	\[
		\Cdim(\bdry G) \asymp_C \log(2m-1) \frac{d l}{|\log d\,|}.
	\]
	In fact, if $G$ is a random group at density $d < \frac{1}{2}$, then a.a.s.
	\[
		\Cdim(\bdry G) \leq C \log(2m-1) \left(\frac{d}{|\log d\,|} \vee \frac{1}{1-2d} \right) l,
	\]
	where $\vee$ denotes the maximum operation.
\end{theorem}
A random group at densities $d<\frac{1}{16}$ is $C'(\frac{1}{8})$, and Theorem~\ref{thm-upper-bound} gives the upper bound in
this range.  At any density $d<\frac{1}{2}$ we have a straightforward 
upper bound for conformal dimension that is linear in $l$ by 
\cite[Proposition 1.7]{Mac-12-random-cdim}.

The lower bound is proven using `round trees' as in~\cite[Theorem~1.4]{Mac-12-random-cdim}.  
As introduced by Gromov, a \emph{round tree} is a CAT($-1$) $2$-complex $A$
which admits an isometric $S^1$ action with a single fixed point, so that
there is an isometrically embedded tree $T$ which meets every fibre of the action
at a single point~\cite[7.C$_3$]{Gro-91-asymp-inv}.
In \cite{Mac-12-random-cdim}, a combinatorial version of a (sector of) such a round tree was
built in the Cayley complex of some small cancellation groups.
Only minor modifications of this approach are required to find our lower bound
at densities $d<\frac{1}{16}$ (see Theorem~\ref{thm-lower-bound-sc}).
However, it is more challenging to extend the range of densities to $d<\frac{1}{8}$ because such groups may only
be $C'(\frac{1}{4})$.  We use an isoperimetric inequality of Ollivier and techniques of Ollivier--Wise to overcome this
obstacle (see Section~\ref{sec-lower-density}).

	As discussed by Ollivier~\cite[Section IV.b]{Oll-05-rand-grp-survey},
	it is interesting to ask whether, given a random group $G$ at some density,
	we can detect the value of $d$.
	Theorem~\ref{thm-main-density} gives a partial answer to this question.
\begin{corollary}\label{cor:detect-density}
	Let $P_{d_0,C}$ be the property that a hyperbolic group $G$ has
	\[ \frac{\log\chi(G)}{\Cdim(\bdry G)} \asymp_C |\log d_0|, \]
	where $\chi(G)$ is the Euler characteristic of $G$.
	There exists $C>1$ so that 
	for any $m\geq 2$, at any density $d_0<\frac{1}{8}$ a.a.s.\ 
	a random $m$-generated group has $P_{d_0,C}$.
	Consequently, if $0<d<d_0^{C^2}$ or $d_0^{1/C^2}<d<\frac{1}{8}$
	then a.a.s.\ a random group at density $d$ does not have $P_{d_0,C}$.	
\end{corollary}
\begin{proof}
	As Ollivier observes, a random, $m$ generated group $G$ at density 
	$d \in (0,\frac{1}{2})$ has 
	$\chi(G) = 1-m+(2m-1)^{dl}$, thus $\log(\chi(G)) = ld \log(2m-1) (1+o(1))$.
	Therefore, by Theorem~\ref{thm-main-density}, at densities $d<\frac{1}{8}$ 
	we have $\log\chi(G)/\Cdim(\bdry G) \asymp_C |\log d\,|$,
	at the cost of multiplying $C$ by $1+o(1)$.
	The second statement then follows.
\end{proof}
A complete answer to Ollivier's question at densities $d<\frac{1}{8}$ 
would follow if $C$ in Theorem~\ref{thm-main-density} could be chosen 
so that $C=C(l) \ra 1$ as $l \ra \infty$.
Possibly one might need to change $d / |\log d\,|$ to a different (but necessarily comparable) function of $d$.

Some restriction to low densities is necessary because our
constructions use small cancellation style arguments, which
get increasingly difficult as the density grows towards $\frac{1}{4}$,
and completely fail at densities $>\frac{1}{4}$ 
(cf.\ \cite{Mac-Prz-14-walls}).

Leaving the density model, we now consider the few relator model of
a random group, where we get even sharper estimates.
\begin{theorem}\label{thm-main-few-rel}
	If $G$ is a random $m$ generator, $n$ relator group, then a.a.s.
	\[
		2 - \frac{5 \log \log l}{\log l} \leq \Cdim(\bdry G) \leq 2+\frac{2 \log\log l}{\log l}.
	\]
\end{theorem}
Since here the conformal dimension is roughly two, while at positive density it goes to infinity, 
we consider random groups where the number of relations grows subexponentially.
(Or from another point of view, $d \ra 0$ as $l \ra \infty\,$;
compare the ``low-density randomness'' of Kapovich--Schupp~\cite{KS-08-low-density-random}.)
It turns out that letting the number of relations grow polynomially lets us tune the
conformal dimension to any value we like.
\begin{theorem}\label{thm-main-polygrowth}
	Fix $m \geq 2$, $K \geq 0$, and $C>0$.
	Then a random $m$ generator, $n=C l^K$ relator group with all relations of length $\leq l$
	satisfies, a.a.s.,
	\[
		2+K - \frac{5 \log \log l}{\log l} \leq \Cdim(\bdry G) \leq 2+K+\frac{2(K+1)\log\log l}{\log l}.
	\]
\end{theorem}
Observe that Theorem~\ref{thm-main-few-rel} follows immediately from the case $K=0$.
These groups are hyperbolic, small cancellation and have Menger sponge boundaries by Champetier~\cite{Cha-95-rand-grps},
so again conformal dimension is essential for distinguishing their quasi-isometry classes.

As before, the upper bound in Theorem~\ref{thm-main-polygrowth} follows 
from Theorem~\ref{thm-upper-bound}.
For the lower bound we again build a round tree $A$ in $G$, as in 
\cite[Theorem 5.1]{Mac-12-random-cdim}.
This round tree is built inductively by adding on $2$-cells in layers
as we travel away from the identity (see Figure~\ref{fig-extend-round-tree}).
Formerly we added these $2$-cells one at a time, 
but here we build a larger round tree 
in $G$ by adding many $2$-cells at once.
This corresponds to solving a perfect matching problem and
is discussed in Section~\ref{sec-new-lower-sc}.

Bourdon~\cite{Bou-97-GAFA-exact-cdim} has calculated the exact conformal dimensions of a family
of Fuchsian buildings, and the values of these dimensions take a dense set of values in $(1, \infty)$.
Theorem~\ref{thm-main-polygrowth} gives the only other way I know of showing the existence
of groups with conformal dimension arbitrarily close to any real number in $[2, \infty)$.

We observe one other consequence of Theorem~\ref{thm-main-polygrowth}.
\begin{corollary}\label{cor:randpolygrowth:qi}
	There is a countable set $Q \subset \R$ so that if $K \in [0,\infty) \setminus Q$,
	and $C >0$ is constant, then as $l \ra \infty$
	random $m$ generator, $n=C l^K$ relator groups pass through infinitely many different quasi-isometry classes.
\end{corollary}
\begin{proof}
	There are only countably many hyperbolic groups, so let $Q$ be the set of $K \in \R$
	so that $2+K$ is a possible value of the conformal dimension.
	By Theorem~\ref{thm-main-polygrowth}, random $m$ generator, $n=C l^K$ relator groups have
	conformal dimension converging to $2+K$, but this cannot be the conformal dimension of any of the (finitely many)
	presentations considered at each length.  Thus the conformal dimensions keep changing as $l \ra \infty$,
	and so the quasi-isometry class of the groups change also.
\end{proof}
It seems plausible that the same result should hold for all $K > 0$, and perhaps for few relator groups ($K=0$) as well.

\subsection{Outline}
In Section~\ref{sec-bourdon-kleiner} we recall Bourdon and Kleiner's upper bound
for conformal dimension.
Ollivier and Wise's walls are modified to define branching walls in 
Sections \ref{sec-wise-walls} and \ref{sec-elem-complexes}.
We show that each branching wall satisfies the Bourdon--Kleiner condition in
Section \ref{sec-complexes-decompose}.
We then apply their result to small cancellation and random groups in 
Section~\ref{sec-new-upper}.

New lower bounds for the conformal dimension of random groups are given in
Sections~\ref{sec-new-lower-sc} and \ref{sec-lower-density}
for the few/polynomial relator model and density model, respectively.

\subsection{Notation}
We write $A \preceq_C B$ for $A \leq CB$, where $C >0$, and write $A \asymp_C B$
if $A \preceq_C B$ and $A \succeq_C B$.
We omit the $C$ if its precise value is unimportant.

\subsection{Acknowledgements}
I gratefully thank Marc Bourdon for describing to me his work with Bruce Kleiner,
and Piotr Przytycki for many interesting conversations about random groups
and walls.
I also thank the referee(s) for many helpful comments.
The author was partially supported by EPSRC grant ``Geometric and analytic aspects of infinite groups''.

\section{Bourdon and Kleiner's upper bound}\label{sec-bourdon-kleiner}

In this section we describe a result of Bourdon and Kleiner which
gives an upper bound for the conformal dimension of certain Gromov hyperbolic
$2$-complexes.  

A combinatorial path (or loop) in a graph is a map from $[0,1]$ (or $\Sph^1$)
to the graph which follows a finite sequence of edges (and has the same
initial and terminal vertex).
A \emph{combinatorial $2$-complex} $X$ is a $2$-complex which is
built from a graph $X^{(1)}$ (the \emph{$1$-skeleton} of $X$)
by attaching $2$-cells,
where the attaching maps $\Sph^1 \ra X^{(1)}$
are combinatorial loops.
All $2$-complexes we consider will be combinatorial.

Given such a space $X$, we give it a geodesic metric by making all edges isometric
to $[0,1]$, and all $2$-cells isometric to regular Euclidean polygons with the
appropriate number of sides.

The \emph{perimeter} $|\partial R|$ of a $2$-cell (a face) $R$ in $X$ is the number of edges adjacent to the face.
The \emph{thickness} of a $1$-cell (an edge) in $X$ is the number of $2$-cells adjacent to
the edge.  The \emph{degree} of a vertex in $X$ is the number of adjacent edges to the vertex.
We say $X$ has \emph{bounded geometry} if
there is a uniform bound on the perimeter, thickness and degrees of cells in $X$.

There are two natural metrics on the $1$-skeleton $X^{(1)}$: the restriction of the
metric on $X$ and the natural path metric on the graph $X^{(1)}$.  If $X$ has bounded
geometry, these two metrics are comparable; we denote both by $d$.

Throughout this paper, we make the following assumption.
\begin{assumption}\label{assump-main}
	$X$ is a connected, simply connected, combinatorial $2$-complex,
	with a geodesic metric making each $2$-cell a regular Euclidean polygon. 
	$X$ is Gromov hyperbolic and has bounded geometry.
	All closed $2$-cells are embedded, and the intersection of any two $2$-cells
	is a connected (possibly empty) set.
\end{assumption}
For background on Gromov hyperbolicity, see \cite[Chapter III.H]{BH-99-Metric-spaces}.

The key example to bear in mind is when $X$ is a Cayley $2$-complex 
of a Gromov hyperbolic group $G$.
Such spaces $X$ do not always satisfy the last sentence of Assumption~\ref{assump-main},
but if $G$ is given by a $C'(\frac{1}{6})$ presentation, then $X$ satisfies
Assumption~\ref{assump-main} (see Lemma~\ref{lem:small-canc-assumption}).
If $G$ is a random group at density $d<\frac{1}{4}$ then the assumption
follows by \cite[Proposition 1.10, Corollary 1.11]{Oll-Wis-11-rand-grp-T}.

The \emph{boundary at infinity $\bdry X$} of $X$, or \emph{visual boundary} of $X$, 
is a compact metric space associated to $X$,
canonically defined up to ``quasisymmetric'' 
homeomorphism~\cite{Pau-96-qm-qi,BS-00-gro-hyp-embed}.

The \emph{conformal dimension} $\Cdim(\bdry X)$ is the infimal
Hausdorff dimension among all Ahlfors regular metric spaces quasisymmetric to $\bdry X$.
(Another common variation on this definition, a priori with a smaller value,
does not require the metric spaces considered to be Ahlfors regular.
Our lower bounds via Theorem~\ref{thm-roundtree-cdim} hold for this value too.)
As already mentioned, we do not work with the definition of conformal dimension directly.

There is a topological compactification of $X$ as $X \cup \bdry X$.
Given $E \subset X$, we define the \emph{limit set} of $E$ to be
$\bdry E = \mathrm{Closure}_{X \cup \bdry X}(E) \cap \bdry X$.

In the following definition, a subcomplex $Y \subset X$ is a subspace of $X$ which
is itself a combinatorial $2$-complex.  (The cells in $Y$ need not be cells in $X$.)
\begin{definition}[{\cite[Definition 3.3]{BK-12-lp-bdry}}]\label{def-decomposes}
	A subcomplex $Y \subset X$ \emph{decomposes} $X$ if
	\begin{itemize}
		\item $Y$ is connected and simply connected,
		\item $Y$ is quasi-convex, i.e.\ any $x,y \in Y$ can be
		joined by a path in $Y$ of length comparable to $d(x,y)$.
		\item Each pair $H_1, H_2$ of distinct connected components of $\overline{X \setminus Y}$ has 
			$\bdry H_1 \cap \bdry H_2 = \emptyset$,
		\item Every sequence $\{ H_i\}$ of distinct connected components of 
			$\overline{X \setminus Y}$ subconverges
			in $X \cup \bdry X$ to a point in $\bdry X$.
	\end{itemize}
	
	A collection of subcomplexes $\{Y_j\}_{j \in J}$ \emph{fully decomposes} $X$ if
	\begin{itemize}
		\item every $Y_j$ decomposes $X$, and
		\item for every $z_1 \neq z_2$ in $\bdry X$, there exists $Y \in \{Y_j\}_{j \in J}$ so that
		for every connected component $E$ of $\overline{X \setminus Y}$, 
		$\{ z_1, z_2 \} \nsubseteq \bdry E$.
	\end{itemize}
\end{definition}

\begin{definition}[{\cite[Definition 2.2]{BK-12-lp-bdry}}]
	An \emph{elementary polygonal complex} is a connected, simply connected combinatorial 
	$2$-complex so that
	\begin{itemize}
		\item every $2$-cell has an even perimeter at least $6$,
		\item every pair of $2$-cells shares at most a vertex or an edge,
		\item the edges of the $2$-cells are coloured alternately black and white, and
		\item every white edge has thickness $1$, every black edge has thickness $\geq 2$.
	\end{itemize}
\end{definition}
See Figure~\ref{fig-Vtree} later for an illustration.

The following theorem is a combination of various results by Bourdon and 
Kleiner~\cite{BK-11-coxeter,BK-12-lp-bdry},
which also incorporates work of 
Bourdon--Pajot~\cite{BP-03-lp-besov}, 
Keith--Kleiner~\cite{KK-XX-modcutpoints}
and Carrasco-Piaggio~\cite{Car-13-conformal-gauge}.
\begin{theorem}[Bourdon--Kleiner]\label{thm-bourdonkleiner}
	Suppose $X$ admits a cocompact isometric group action and has connected boundary $\bdry X$.
	
	Suppose $X$ is fully decomposed by a family $\{Y_j\}_{j\in J}$ of
	elementary polygonal complexes,
	each of the $2$-cells of which has perimeter in $[2m, C]$, for fixed $m \geq 3$ 
	and $C< \infty$, and whose black edges have thickness in $[2, k]$, for some fixed $k \in \N$.  Then
	\[
		\Cdim(\bdry X) \leq 1 + \frac{\log(k-1)}{\log(m-1)}.
	\]
\end{theorem}
\begin{proof}
	We use the terminology of Bourdon and Kleiner; see their papers
	for more details.
	By \cite[Proposition 3.3]{BK-11-coxeter}, $\bdry X$ is approximately self-similar.
	So \cite[Theorem 3.8(1)]{BK-12-lp-bdry} gives that $\Cdim(\bdry X) = p_{sep}(X)$, where
	$p_{sep}(X)$ is the infimal $p \geq 1$ so that functions in $\ell_p H^1_{\mathrm{cont}}(X)$
	can distinguish any two points in $\bdry X$.
	The conclusion then follows from the upper bounds on $p_{sep}(X)$ given by
	\cite[Corollaries 3.6(2), 6.6]{BK-12-lp-bdry}.
\end{proof}

\section{Walls and diagrams}\label{sec-wise-walls}

Our goal is to find embedded elementary polygonal complexes in $X$
that fully decompose $X$ and whose $2$-cells have large perimeters.
The way in which we will accomplish this is to build suitable rooted trees in $X$,
and then take small neighbourhoods of these trees (see Figure~\ref{fig-Vtree}).

Wise~\cite{Wise-04-cubulating-sc-groups} and Ollivier--Wise~\cite{Oll-Wis-11-rand-grp-T}
built walls in the Cayley complexes of both small cancellation groups
and random groups at densities $< \frac{1}{5}$.
Each wall is built by taking an edge in $X$, joining the midpoint of the edge to
the midpoint of each antipodal edge, joining each of these to the midpoints of edges
antipodal to them, and so on in this way (see Definition~\ref{def-Ipathwall}).
Ollivier and Wise show that if certain disc diagrams do not exist, then such walls
will be quasi-convex and embedded in $X$.

A small neighbourhood of such a wall is almost an elementary polygonal complex,
but the $2$-cells only have perimeter four.  
In this and the following sections we adapt Ollivier and Wise's construction to
build branching walls (see Section~\ref{sec-elem-complexes}),
neighbourhoods of which will fully decompose $X$.

Although, as in \cite[Section 3]{Oll-Wis-11-rand-grp-T}, it is possible to deal
with general combinatorial $2$-complexes, we restrict our attention
to the following special case:
Throughout this and following sections, $X$ is the \emph{Cayley complex} of a group 
presentation $G = \langle \cS | \cR \rangle$, where each 
relation $r \in \cR$ is a cyclically reduced word in the generators $\cS$.
Recall that $X$ is the universal cover of the $2$-complex $Y$
which is formed by taking a bouquet
of $|\cS|$ oriented, labelled loops, one for each element of $\cS$, and attaching a disc for each 
$r \in \cR$ along the path labelled by~$r$.

In the case that a relation $r \in \cR$ is a proper power, we modify this construction slightly.
If we write $r = u^i$ for a maximal $i \in \N$, then in $X$ we see many bundles 
of $i$ discs with identical boundary paths; we collapse each of 
these bundles of $i$ discs into a single disc.

\subsection{Disc diagrams}
For general references on disc diagrams and small cancellation theory, we refer the
reader to \cite{Lyndon-Schupp-small-canc,McC-Wis-02-fans-ladders}.
A \emph{combinatorial map} $D \ra X$ of combinatorial $2$-complexes is a continuous map
so that its restriction to any open cell of $D$ is a 
homeomorphism onto an open cell of $X$~\cite[Definition~2.1]{McC-Wis-02-fans-ladders}.
\begin{definition}\label{def-discdiagram}
	A \emph{disc diagram} $D \ra X$ is a contractible, finite, pointed (combinatorial) 
	$2$-complex $D$ with a combinatorial map $D \ra X$ and 
	a specific embedding in the plane so that the base point lies on the boundary
	$\partial D \subset \R^2$.
	
	Note that the map $D \ra X$ gives each edge of $D$ 
	an orientation and a labelling by an element of $\cS$,
	so that for each $2$-cell $R \subset D$, reading the edge labels along the boundary 
	$\partial R$ of $R$ gives (a cyclic conjugate of) a word $r \in \cR$
	or its inverse.
	
	If $w$ is the word in $G$ given by reading $\partial D$ counter-clockwise from
	the base point, we say that $D$ is a van Kampen diagram for $w$.  
	(Equivalently, the \emph{boundary path} $\partial D \ra X$ is labelled by $w$.)
	We write $|\partial D|$ for the length of this word.
	
	A \emph{cancellable pair} in $D \ra X$ is a pair of distinct $2$-cells 
	$R_1, R_2 \subset D$ which meet along at least one edge in $X$, 
	so that the boundary paths of $R_1, R_2$ starting from this edge map to the same
	paths in $X$.
	If $D$ has no cancellable pairs, it is \emph{reduced}.
	
	We write $|D|$ for the number of $2$-cells in $D$.
\end{definition}
This is a slight rewording of the usual definition of a van Kampen diagram to match 
the language of \cite{Oll-Wis-11-rand-grp-T}.

A $2$-cell $R$ in a disc diagram $D \ra X$ is \emph{external} if 
$\partial R \cap \partial D \neq \emptyset$, otherwise it is \emph{internal}.

\subsection{Ollivier and Wise's walls}\label{ssec:ow-I-walls}
We now recall Ollivier and Wise's construction of walls in $X$,
subject to the non-existence of certain disc diagrams.
For the purposes of this subsection, we subdivide edges so that 
the perimeter of every $2$-cell in $X$ is even.
\begin{definition}\label{def-Ipathwall}
	An \emph{I-path} is an immersed path $\lam$ in $X$ which meets $X^{(1)}$ only in 
	the midpoints of edges, and locally crosses each $2$-cell from the middle of an 
	edge straight to the middle of the antipodal edge.
	
	An \emph{I-wall} $\Lambda \ra X$ is the maximal union of I-paths containing a 
	given midpoint of an edge in $X^{(1)}$, identified locally so that 
	$\Lambda$ is a tree with an	immersion $\Lambda \ra X$.
\end{definition}
I-walls are called `hypergraphs' by 
Ollivier--Wise; we call them `I-walls' to compare them with the branching walls
we consider later.
We also call their various types of `collared diagrams' 
I-collared diagrams to distinguish them from our V-collared diagrams.
\begin{definition}[{\cite[Definition 3.2]{Oll-Wis-11-rand-grp-T}}]
	An \emph{I-collared diagram} is a disc diagram $D \ra X$
	with external 2-cells $R_1, \ldots R_n$, $n \geq 2$, where $R_1$ is called a \emph{corner},
	and which contains an I-path $\lam$ passing through $R_2, \ldots, R_n$ in a single
	segment, and self-intersecting in $R_1$.  (See Figure~\ref{fig-Icollared}.)
	It is \emph{reduced} if it is reduced in the usual sense.
\end{definition}
\begin{figure}
	\centering
	\def\svgwidth{0.9\columnwidth}
	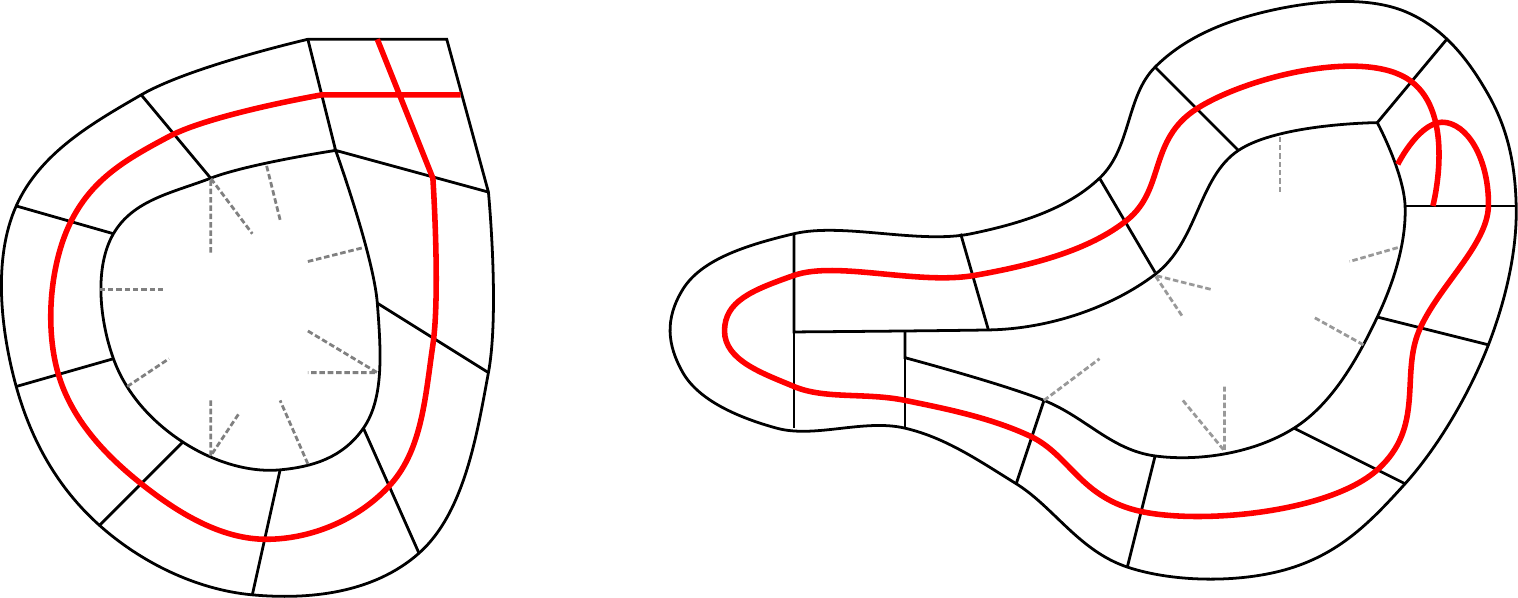
	\caption{I-collared diagrams (a) and (b)}\label{fig-Icollared}
\end{figure}
An I-wall comes with a natural immersion $\Lambda \ra X$;
Ollivier--Wise find conditions which make this immersion an embedding.
\begin{theorem}[{\cite[Theorem 3.5]{Oll-Wis-11-rand-grp-T}}]\label{thm-Iwall-embed}
	If there are no reduced I-collared diagrams for $X$, then every I-path embeds,
	and hence every I-wall embeds.
\end{theorem}
Before we sketch the proof of this theorem, we need another definition.

Suppose $\lambda \subset X$ is an I-path which goes through 
$2$-cells $R_1, R_2, \ldots, R_n \subset X$
in turn.  Let $R_i'$ be a copy of $R_i$ for each $i$,
and let $L$ be the combinatorial $2$-complex given by identifying the
boundaries of $R_i'$ and $R_{i+1}'$ along the single edge corresponding to 
$\lambda \cap R_i \cap R_{i+1}$.
We call $L$ the \emph{ladder} 
of $\lambda$; it comes with a natural
combinatorial map $L \ra X$ extending the immersion $\lambda \ra X$.
If we further identify the boundaries of $R_i'$ and $R_{i+1}'$ along the
all the edges corresponding to $R_i \cap R_{i+1}$, we get a $2$-complex
$\check{L}$ called the \emph{carrier} of $\lambda$, which also has a natural
combinatorial map $\check{L} \ra X$.

If an I-path $\lambda \subset X$ does not embed, then its ladder $L \ra X$ will
have (say) the first and last $2$-cells $R_1, R_n$ mapping to the same $2$-cell of $X$.
Assumption~\ref{assump-main} and control on local behaviour like 
Definition~\ref{def-Vpath-admissible} below gives that $n \geq 3$.
Let $A = L / \!\sim$ be the quotient of $L$ by identifying $R_1 \ra X$ and $R_n \ra X$.
Topologically, $A$ is an annulus or a M\"obius strip.
As $A$ is homotopic to $\Sph^1$, we can find 
a simple, non-contractible cycle $P \ra A$, and let
$D \ra X$ be a disc diagram with boundary path $P$.
The combinatorial $2$-complex $F = A \cup_{P = \partial D} D$ 
is a \emph{quasi-I-collared diagram} $F \ra X$.

It may be that the planar embedding of $D$ can be extended to a planar embedding of $F$,
but this will not generally be possible (cf.\ Figure~\ref{fig-cancelloop}).
We say $F \ra X$ is \emph{reduced} if there are no cancellable pairs, as in
Definition~\ref{def-discdiagram}.

Theorem~\ref{thm-Iwall-embed} is proved as follows. 
First, if an I-path doesn't embed, then it bounds a quasi-I-collared 
diagram \cite[Lemma 3.8]{Oll-Wis-11-rand-grp-T}.
One then performs reductions to find a reduced quasi-I-collared diagram
\cite[Lemma 3.9]{Oll-Wis-11-rand-grp-T}.
Finally, from a reduced quasi-I-collared diagram, 
one extracts a reduced I-collared diagram \cite[Lemma 3.10]{Oll-Wis-11-rand-grp-T}, 
contradicting our assumption.

We will use a similar outline frequently in what follows.

\subsection{V-paths}
To build elementary polygonal complexes, we need a generalisation of I-paths
where the paths can bend as they cross $2$-cells in $X$.
\begin{definition}
	A \emph{V-path} of length $l(\alp) \in \N$ is an (oriented) 
	immersed path $\alp: [0, l(\alp)] \ra X$ so that
	for $i = 0, \ldots, l(\alp)-1$,
	$\alp(i)$ is a midpoint of an edge of $X^{(1)}$, 
	$\alp$ restricted to $[i, i+\frac12]$ is a straight segment
	joining $\alp(i)$ to the centre $\alp(i+\frac12)$ of a $2$-cell $R$ adjacent to
	$\alp(i)$, and 
	$\alp$ restricted to $[i+\frac12,i+1]$ is a straight segment
	joining $\alp(i+\frac12)$ to the midpoint $\alp(i+1) \neq \alp(i)$ of 
	an edge of $R$.
\end{definition}
Since $\alp$ is immersed, for any V-path $\alp$ of length at least two we have 
$\alp(i-\frac12) \neq \alp(i+\frac12)$ for $i=1, \ldots, l(\alp)-1$.

We cannot hope to control V-paths in general, so we restrict to
a suitably rich collection of V-paths.
\begin{definition}\label{def-Vpath-admissible}
	Suppose $\cV_0$ is a collection of oriented V-paths in $X$ of length one,
	so that for every $2$-cell $R \subset X$ and edge $e$ in $\partial R$,
	there exists $\alp \in \cV_0$ so that $\alp(0) \in e$ and $\alp(\frac12) \in R$.
	
	Let $\cV$ be the collection of all V-paths $\alp$ so that for each $0 \leq i < l(\alp)$,
	$\alp$ restricted to $[i,i+1]$ lies in $\cV_0$.
	We call each $\alp \in \cV$ a \emph{$\cV$-path}.
	
	We say such a $\cV$ is \emph{crossing} if
	whenever $\alp \in \cV$ has length three and goes through 
	$2$-cells $R_1, R_2, R_3$ in order then $R_1 \cap R_2 \cap R_3 = \emptyset$.
	
	In addition, we say $\cV$ is \emph{reversible} if 
	for all $\alp \in \cV$ we have $-{\alp} \in \cV$, where
	$-{\alp}$ denotes the V-path $\alp$ with reversed orientation.
\end{definition}

For example, the collection of all I-paths considered by Ollivier--Wise 
is a reversible and crossing collection of V-paths.
From now on, we will assume that there is a fixed reversible and 
crossing collection of V-paths $\cV$,
and all V-paths considered will be in $\cV$.

Given a V-path $\alpha$, we 
define the \emph{ladder $L \ra X$} and \emph{carrier $\check{L} \ra X$}
of $\alpha$ exactly as in Subsection~\ref{ssec:ow-I-walls}.
\begin{remark}
	The crossing assumption ensures that if $\alp$ is a 
	 $\cV$-path in $X$,
	then the carrier $\check{L} \ra X$ of $\alp$ will be planar.
\end{remark}

We now define (quasi-)V-collared diagrams, in analogy to the \mbox{(quasi-)} I-collared diagrams
mentioned above (compare \cite[Definition 3.11]{Oll-Wis-11-rand-grp-T}).
These are disc diagrams ``collared'' by attaching
ladders of V-paths 
around the boundary of the diagram, in a not-necessarily-planar way.
\begin{definition}\label{def-quasi-V-collared}
	Let $n \geq 1$ be an integer, and write $\{1,2,\ldots,n\}$ as a disjoint union $I \sqcup J$.
	Suppose that for $i \in I$, $\alp_i$ is an oriented 
	V-path of length at least $2$, with ladder 
	$L_i \ra X$.
	Let $P_i$ be a path immersed in $L_i^{(1)}$, joining a point in
	the boundary of the first $2$-cell in $L_i$ to one in the last $2$-cell of $L_i$.
	For $j \in J$, let $P_j$ be any path immersed in $X^{(1)}$.
	We require that (with subscripts $\mod n$):
	\begin{enumerate}
		\item 
		If $i \in I$ and $i+1 \in I$, then the last $2$-cell of $L_i$ and the
		first $2$-cell of $L_{i+1}$ have the same image in $X$.
		If $n \geq 2$ then $\alp_i(l(\alp_i)-1)$ and $\alp_{i+1}(1)$ 
		are distinct in $X$ (i.e., no doubling back).
		\item 
		The last point of $P_i$ and the first point of $P_{i+1}$ coincide in $X$.
		\item
		If $i \in J$ then $i +1 \in I$ and $i-1 \in I$.
	\end{enumerate}
	The second condition lets us define a cyclic path $P = \bigcup_i P_i$.
	Let $D \ra X$ be a disc diagram with boundary path $P$.
	
	Let $A' = \bigsqcup_{i \in I} L_i$, and let $A \ra X$ be the quotient of $A'$
	where whenever $i,i+1 \in I$ we identify the last $2$-cell of $L_i$ with the first
	$2$-cell of $L_{i+1}$.
	
	A diagram \emph{quasi-V-collared by the $\alp_i$, $i\in I$, and the $P_i$, $i \in J$},
	is the union $E = A \cup_P D \ra X$.
	The diagram is \emph{V-collared by the $\alp_i$, $i\in I$, and the $P_i$, $i \in J$},
	if the planar embedding of $D$ can be extended to a planar embedding of $E$.
	
	The \emph{corners} of $E$ are the initial and final $2$-cells of each $L_i$, $i \in I$.
\end{definition}
If we just say that a diagram is {(quasi-)V-collared}, we assume that $n=1$ and $I=\{1\}$.
If we want to specify that the collaring V-paths lie in a collection $\cV$,
we say a diagram is (quasi-)$\cV$-collared.

V-paths are very similar to I-paths, but with turns as they cross each face,
and similar arguments to those of Theorem~\ref{thm-Iwall-embed} give the following.
\begin{theorem}\label{thm-v-wall-embed}
	If 
\begin{gather}\label{eq:as-no-red-V-collared}
	\parbox{0.9\textwidth}{there is no reduced $\cV$-collared diagram} \tag{A}
\end{gather} 
then every $\cV$-path embeds.  
\end{theorem}
We keep track of the assumptions we are making for future reference.
\begin{proof}
Suppose $\alp$ in $X$ is a $\cV$-path which does not embed.
We split the proof into three lemmas which use $\alp$ to
find a contradiction to \eqref{eq:as-no-red-V-collared}. 
\begin{lemma}\label{lem:Vpathembed1}
	There is a quasi-$\cV$-collared diagram.
\end{lemma}
\begin{proof}
	By restricting $\alp$ to a shorter $\cV$-path, we can assume that 
	$\alp$ only self-intersects in the first and last $2$-cells it crosses.
	As every $\cV$-path of length at most three embeds by 
	Definition~\ref{def-Vpath-admissible},
	we know $\alp$ has length at least four.
	
	Let $L \ra X$ be the ladder of $\alp$, and let $A$ be
	the quotient of $L$ under the identification of its first and last $2$-cells;
	recall that $A$ is homeomorphic to an annulus or a M\"obius strip.
	Let $P$ be a path in $L$ joining the boundaries of its first and last $2$-cells,
	so that $P$ forms a closed immersed path in $A$.
	As in Definition~\ref{def-quasi-V-collared}, let 
	$E \ra X$ be a diagram quasi-V-collared by $\alp$,
	where 
	$D \ra X$ is a disc diagram with boundary path $P$, and
	$E = A \cup_P D$.
\end{proof}
Let $F = A \cup_P D$ be a quasi-$\cV$-collared diagram of minimal area,
among all self-intersecting $\cV$-paths $\alp$,
and all $A \cup_P D$ as in Lemma~\ref{lem:Vpathembed1}.
\begin{lemma}\label{lem:Vpathembed2}
	$F$ is a reduced quasi-$\cV$-collared diagram.
\end{lemma}
\begin{proof}
	Suppose there is a cancellable pair of $2$-cells $R_1, R_2$ in $F$.
	
	First, if $R_1,R_2 \subset D$, then do the usual van Kampen diagram reduction where we
	remove the two open $2$-cells and the intersection of their closures,
	then identify the corresponding remaining boundaries.
	This gives a quasi-$\cV$-collared diagram of smaller area, contradicting
	the assumption on $F$.
	
	Second, if $R_1 \subset A$ and $R_2 \subset D$, then
	we push $P$ to the other side of $R_1$, and then identify the two faces $R_1, R_2$ 
	in the diagram.  This has the effect of removing $R_2$ from $D$,
	as shown in the two examples in Figure~\ref{fig-cancelAD}.
	As this again reduces the area of $F$, it is impossible.
\begin{figure}
	\centering
	\def\svgwidth{0.9\columnwidth}
	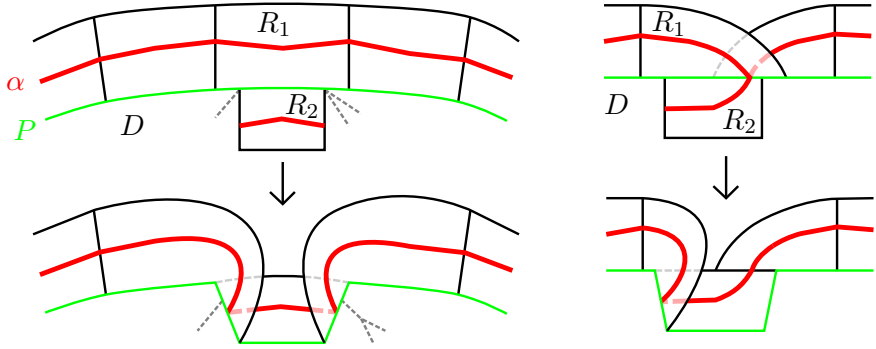
	\caption{Cancelling $2$-cells between $A$ and $D$}\label{fig-cancelAD}
\end{figure}
\begin{figure}
	\centering
	\def\svgwidth{0.85\columnwidth}
	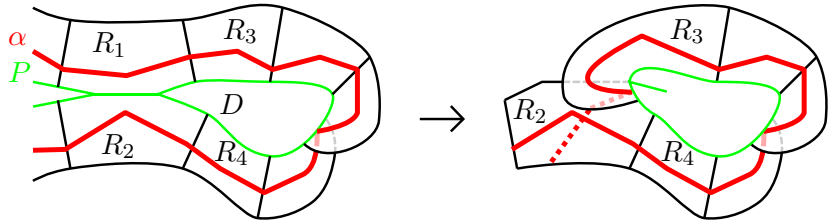
	\caption{Cancelling $2$-cells between $A$ and $A$ (a) and (b)}\label{fig-cancelAA}
\end{figure}

	Finally, suppose we have a cancellable pair $R_1, R_2 \subset A$,
	as in Figure~\ref{fig-cancelAA}(a).
	The $\cV$-path $\alp$ travels from $R_1$ through a sequence of $2$-cells
	$R_3, \ldots, R_4$ to $R_2$, avoiding the corner of $E$.  
	Since $\alp$ is an immersion and crossing, this sequence
	has length at least four.
	As in the usual argument, we remove the 
	interiors and intersection of $R_1$ and $R_2$, and identify the corresponding
	boundary paths edge by edge until the edge containing $R_1 \cap R_3 \cap \alp$
	(or $R_2 \cap R_4 \cap \alp$) is identified with an edge formerly in $R_2$ (or $R_1$);
	see Figure~\ref{fig-cancelAA}(b).
	We restrict $\alp$ to the $\cV$-path $\alp'$ from $R_3$ to $R_2$, and
	extend $\alp'$ from $R_3$ (or $R_4$) into $R_2$ to define
	a new quasi-$\cV$-collared diagram.
	(If the edge $R_1 \cap R_3$ is identified with the edge $R_2 \cap R_4$, we can discard
	$R_2$ too, and $\alp'$ forms a loop.)
	Again, this new diagram would have smaller area, which is impossible.
\end{proof}
\begin{lemma}\label{lem:Vpathembed3}
	$P$ does not cross $\alp$ in $F = A \cup_P D$.
\end{lemma}
\begin{proof}
	If the path $P$ crosses $\alp$ in $A$,
	then we can extend $\alp$ to a $\cV$-path $\alp'$ that
	travels into $D$.
	If $\alp'$ self-intersects, then a sub-$\cV$-path $\alp''$ of $\alp'$ collars
	a reduced subdiagram $F'$ of $D$, see Figure~\ref{fig-cancelloop}(a).
	But this subdiagram has smaller area than $F$, which is impossible.
	\begin{figure}
	\centering
	\def\svgwidth{0.75\columnwidth}
	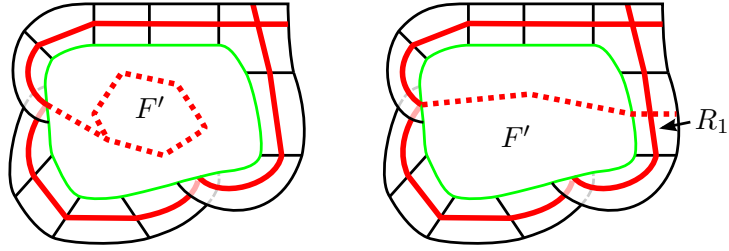
	\caption{Finding smaller quasi-collared diagrams (a) and (b)}\label{fig-cancelloop}
	\end{figure}
	
	If $\alp'$ does not self-intersect, it must intersect $A$ again in a
	$2$-cell $R_1$.  The part of $F$ which is bounded by $\alp''$ and does not
	contain the corner of $F$ gives us a smaller quasi-$\cV$-collared diagram $F'$
	with corner $R_1$, see Figure~\ref{fig-cancelloop}(b), which is impossible.
\end{proof}
Because $P$ does not cross $\alp$, it follows that $A$ is an annulus and 
$F = A \cup_P D$ is collared, contradicting assumption~\eqref{eq:as-no-red-V-collared}.
Therefore Theorem~\ref{thm-v-wall-embed} is proved.
\end{proof}

We now give conditions which ensure that every $\cV$-path $\alp \ra X$ is uniformly
quasi-isometrically embedded in $X$.  
\begin{proposition}\label{prop-V-path-qi}	
	Suppose we have \eqref{eq:as-no-red-V-collared}, and
\begin{gather}
	\parbox{0.9\textwidth}{there is no reduced diagram 
		collared by a $\cV$-path of length $\geq 3$ and a trivial path, and} 
	\tag{B}\label{eq:as-no-red-V-triv-collared} \\
	\label{eq:as-V-path-geod-in-ladder}\tag{C}
	\parbox{0.9\textwidth}{
	if $\alp$ is a $\cV$-path with $l(\alp)\geq 3$ and ladder $L \ra X$,
	and $\gam \subset X^{(1)}$ is a geodesic,
	and if $E \ra X$ is a reduced diagram collared by $\alp$ and $\gam$, 
	then $\gam \subset L$ in $E$.}
\end{gather}
	Then the carrier of every $\cV$-path embeds in $X$, 
	and the $\cV$-paths are uniformly quasi-convex in $X$:
	the endpoints of any $\cV$-path of length $n$ are at least $n/6$ apart in $X^{(1)}$.
	In particular, $\cV$-paths are uniformly quasi-isometrically embedded in $X$.
\end{proposition}
Note that in assumption~\eqref{eq:as-V-path-geod-in-ladder}, necessarily the
two endpoints of $\gam$ lie in the boundaries of the initial and final $2$-cells
of $L$, respectively.
\begin{proof}
	Suppose $\alp$ is a $\cV$-path, with ladder $L = R_1 \cup \cdots \cup R_n \ra X$,
	and carrier $\check{L} \ra X$ which has the same $2$-cells identified along additional
	edges.  Note that $L \ra X$ factors through $L \ra \check{L} \ra X$.
	
	First, we show that carriers of $\cV$-paths embed.
	If $n \leq 3$ this follows from Assumption~\ref{assump-main} and the fact that 
	$\cV$-paths are crossing, so we assume that $n > 3$.
	In fact, these assumptions imply that $\check{L}\ra X$ is an immersion,
	and by Theorem~\ref{thm-v-wall-embed} $\alp$ embeds,
	so we just have to show that $R_i \cap R_j = \emptyset$ in $X$ for $|i-j|>1$.
	
	It suffices to show that $R_1 \cap R_n = \emptyset$ in $X$ for $n > 3$.
	If not, there is a diagram which is quasi-collared by a $\cV$-path $\alp$ and a
	trivial path.  In other words, we have a diagram $E \ra X$
	such as in Figure~\ref{fig-Vladder}(a), with $E = L \cup D$ for a disc diagram $D \ra X$.
\begin{figure}
	\centering
	\def\svgwidth{0.85\columnwidth}
	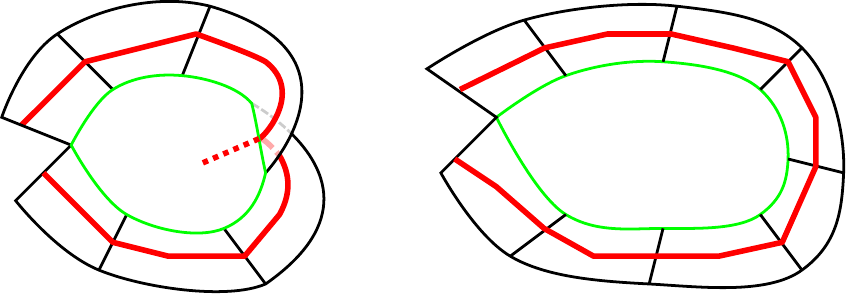
	\caption{(a) and (b)}\label{fig-Vladder}
\end{figure}
	We follow a similar argument to that of Theorem~\ref{thm-v-wall-embed}.
	Assume that $E$ has minimal area among all such diagrams,
	so there is no cancellable pairs within $D$, or between $D$ and $L$.
	(There are no cancellable pairs between $L$ and $L$, for that would give
	a self-intersecting $\cV$-path.)  So $E$ is reduced.
	In fact, $E$ is collared as in Figure~\ref{fig-Vladder}(b),
	for if $P$ crossed $\alp$, 
	we could extend $\alp$ into $D$ to find a smaller diagram collared by
	a $\cV$-path. 
	But such a collared diagram as this contradicts 
	assumptions \eqref{eq:as-no-red-V-collared}, \eqref{eq:as-no-red-V-triv-collared},
	or crossing.

	Second, we show that $\alp$ is quasi-convex in $X$.
	Consider $\alp$ as a locally (and hence globally) Lipschitz 
	map from $[0,n]$ into $X$.
	We show that if
	$\gam$ is a geodesic in $X^{(1)}$ joining the endpoints of $\alp$ in $R_1$ and $R_n$,
	then the length $l(\gam)$ of $\gam$ satisfies $l(\gam) \geq n/6$.
	
	If $n=1$ or $2$ then since $\alp$ has distinct endpoints we have
	$l(\gam) \geq 1 \geq n/6$, so we assume $n \geq 3$.

	Let $E \ra X$ be a diagram quasi-collared by $\alp$ and $\gam$,
	where $E = L \cup_P D$ for a disc diagram $D \ra X$ and a path $P$ which goes from
	$R_1$ to $R_n$ in $L$ and then along $\gam$.
	
	As before, we can reduce any cancellable pairs in $E$ (since $\alp$ is embedded,
	all such reductions will involve one or two $2$-cells in $D$ and will preserve 
	both $L$ and $\gam$).
	Let us again call this diagram $E \ra X$, which could look like 
	Figure~\ref{fig-Vpathqi}.
\begin{figure}
	\centering
	\def\svgwidth{0.85\columnwidth}
	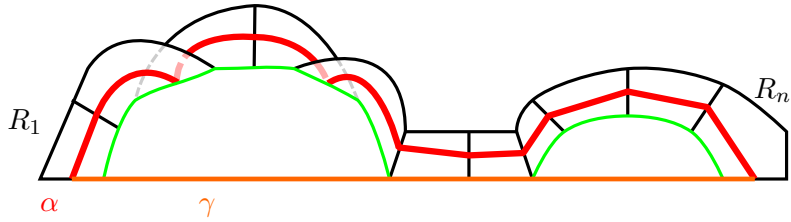
	\caption{Quasi-convexity of ladders}\label{fig-Vpathqi}
\end{figure}
	
	We claim that $\gam$ actually meets every $2$-cell in the ladder.
	By Definition~\ref{def-Vpath-admissible}, $\alp$ is crossing so
	$\gam$ cannot jump over a $2$-cell in the ladder without leaving the ladder.
	If $\gam$ jumps from some $R_i$ to $R_j$ for $j \geq i+2$, then
	there is a subdiagram $F$ of $E$, quasi-collared by $\alp \cap (R_i \cup \cdots \cup R_j)$
	and a subgeodesic of $\gam$ which meets $R_i \cup \cdots \cup R_j$ only at its endpoints.
	(The boundary path $P'$ for $F$ consists of the subgeodesic of $\gam$, followed
	by the relevant portion of the old path $P$ in $R_i \cup \cdots \cup R_j$.)
	
	This diagram $F$ cannot be collared as it would contradict 
	assumption \eqref{eq:as-V-path-geod-in-ladder} because 
	$\gam \cap R_{i+1} = \emptyset$.
	If $F$ is quasi-collared and $P'$ crosses $\alp$,
	we can extend the $\cV$-path $\alp$
	into $D$ to find a reduced diagram collared by a $\cV$-path $\alp'$ of length at 
	least three, and a geodesic which does not meet every $2$-cell of the ladder of $\alp'$,
	again contradicting \eqref{eq:as-V-path-geod-in-ladder}.
	So our claim holds: $\gam$ meets every $2$-cell in $L$, in order.
	
	For each $2$-cell $R_i$, $i=2, \ldots, n-1$, 
	by the crossing assumption of Definition~\ref{def-Vpath-admissible},
	and the fact that $\check{L}$ is embedded, 
	either $\gam$ contains an edge from $R_i \setminus (R_{i-1} \cup R_{i+1})$
	or $\gam$ contains an edge which jumps from $R_{i-1}$ to $R_i$ or from $R_i$ to $R_{i+1}$.
	In any case, $\gam$ has to contain at least $(n-2)/2$ edges.
	Finally, as $n \geq 3$, $(n-2)/2 \geq n/6$.
\end{proof}
\begin{remark}\label{rmk:geod-meets-every-cell-in-ladder}
	Note that the proof shows that if $\gam \subset X^{(1)}$ is a geodesic which joins the
	endpoints of a V-path $\alp$, then $\gam$ meets every $2$-cell in the ladder
	(or equivalently, the carrier) of $\alp$.  
	Thus if the boundary of any two adjacent $2$-cells in $X$ is 
	geodesically convex in $X^{(1)}$, then (the $1$-skeleton of) any carrier 
	is geodesically convex in $X^{(1)}$.
\end{remark}


\section{Defining branching walls}\label{sec-elem-complexes}

In this section we define a collection of elementary polygonal complexes
in $X$. 
We begin by setting up a general framework, and then describe the specific
complexes we use to prove Theorem~\ref{thm-upper-bound}.

\subsection{$\cV'$-branching walls}
Recall that $\cV$ was a fixed crossing and reversible collection of
V-paths.
We now restrict the collection of V-paths further to a crossing 
collection of V-paths $\cV' \subset \cV$, which need not be reversible.

(Later, $\cV$ will consist of all V-paths $\alp$ so that
$\alp(i)$ and $\alp(i+1)$ are roughly antipodal for each $i$,
while $\cV'$ will be a subset of these paths so that once $\alp(i)$ is fixed,
the possible locations of $\alp(i+1)$ are spread out.)

We begin by defining, for each edge $e \subset X^{(1)}$, an
immersed, rooted, bipartite tree $Z_e \ra X$.
The idea is that we follow out all $\cV'$-paths from the midpoint of $e$.

We build $Z_e \ra X$ by induction, and colour all vertices black or white according
to whether they correspond to the midpoint of an edge in $X^{(1)}$, or the middle
of a $2$-cell in $X$.
First, let $f_0:Z_{e,0} \ra X$ be a single root vertex, coloured black,
mapping to the midpoint of $e$, and we say that this midpoint is \emph{exposed}.

For the inductive step, suppose we are given 
an immersed rooted tree $f_{k}:Z_{e, k} \ra X$ 
and a collection of exposed vertices each coloured black.
For each such point $v$, consider each $2$-cell $R \subset X$
which is adjacent to $f_k(v)$, with the exception of the $2$-cell 
corresponding to the parent of $v$ (if it exists).
Consider all $\cV'$-paths of length one which travel from $f_k(v)$ across $R$,
and define $Z_{e,k+1}$ by adding a white vertex adjacent to $v$ corresponding to
the midpoint of $R$, and black vertices adjacent to this white vertex for the
endpoints of these $\cV'$-paths.  These black vertices are precisely the
exposed vertices of $Z_{e,k+1}$.
The immersion $f_k:Z_{e,k} \ra X$ extends to an immersion $f_{k+1}:Z_{e,k+1} \ra X$
which maps these new edges along the corresponding $\cV'$-paths.

We let $Z_{e}\ra X$ be the union of this increasing sequence of trees
and maps. 

Provided 
\begin{equation}\label{eq:as-white-vertex-deg-at-least-three}\tag{D}
	\parbox{0.9\textwidth}{
	each white vertex $v \in Z_e$ has degree $d(v)$ at least three,	}
\end{equation}
we can define
an elementary polygonal complex $Y_e$ associated to the edge $e \subset X^{(1)}$,
with an immersion $Y_e \ra X$, by taking a small neighbourhood of $Z_e \ra X$.
In effect, this replaces each black vertex with a black edge, and each white vertex $v$
with a $2$-cell with perimeter $2d(v)$, alternating between black edges for the adjacent
black vertices and white edges.
See Figure~\ref{fig-Vtree} for an illustration, where $Z_e$ is red,
$Y_e$ is green, and white edges are dashed.
\begin{definition}\label{def-Vcomplex}
	Given an edge $e \subset X^{(1)}$, we
	call the rooted, bipartite, immersed tree $Z_e \ra X$
	the \emph{$\cV'$-branching wall with root $e$}. 
	We call the corresponding $Y_e \ra X$ 
	the \emph{$\cV'$-elementary polygonal complex with root $e$}, 
	or \emph{$\cV'$-complex} for short.
\end{definition}
\begin{figure}
	\centering
	\def\svgwidth{0.65\columnwidth}
	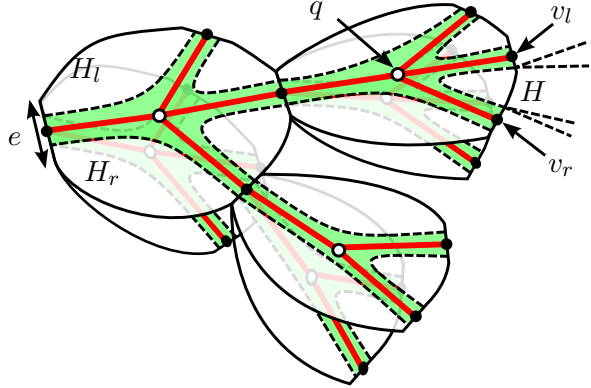
	\caption{Part of an elementary polygonal complex}\label{fig-Vtree}
\end{figure}

\subsection{$\cV$-paths in small cancellation groups}\label{ssec-sc-V-def}

Now suppose $X$ is the Cayley complex of a $C'(\lam)$ group $G$,
where $0 < \lam \leq \frac{1}{8}$ is fixed.
Recall that a (cyclically reduced)
group presentation $\langle S | R \rangle$ is $C'(\lam)$ if
whenever $r, r'$ are distinct cyclic conjugates of relations $R \cup R^{-1}$,
the common initial word of $r$ and $r'$ has length $<\lam \min\{|r|,|r'|\}$.
In this situation, $X$ satisfies Assumption~\ref{assump-main}.
For completeness, we now give a proof of this.
\begin{lemma}\label{lem:small-canc-assumption}
	The Cayley complex $X$ of a $C'(\frac{1}{6})$ finite presentation 
	$\langle S|R \rangle$ satisfies Assumption~\ref{assump-main}.
\end{lemma}
\begin{proof}
	As $X$ is a Cayley complex of a finitely presented group, it
	is a connected and simply connected $2$-complex with bounded geometry.
	As is well known, $C'(\frac{1}{6})$ groups are hyperbolic; 
	indeed, $\langle S|R \rangle$ is a Dehn 
	presentation~\cite[page 244]{Ghys-dlH-90-hyp-groups}.
	
	We now show that for any $2$-cell $R \ra X$, the boundary $\partial R$
	is embedded.
	Suppose this is not true.  Then there is a $2$-cell $R \ra X$
	and a closed non-trivial interval $P \subsetneq \partial R$ so that the induced
	map $P \subset R \ra X$ gives a closed loop $P \ra X$.
	Let $D \ra X$ be a disc diagram with boundary $P \ra X$,
	and assume $D$ has minimal area for all such choices of $P$ and $D$.
	
	Gluing $D$ to $R$ along $P$ gives a disc diagram $E = R \cup_P D \ra X$.
	In fact, because $D$ has minimal area, $E$ is reduced:
	any cancellable pair in $D$ would already be removed,
	and if $R$ and a $2$-cell $R' \subset D$ are a cancellable pair,
	we can remove $R$ and $R'$ from $E$ and glue up their boundaries
	to find a disc diagram $D'$ of smaller area with boundary 
	$\overline{\partial R \setminus P} \ra X$, contradicting minimality of area.
	But now $E = R \cup_P D \ra X$ is a reduced disc diagram for a
	$C'(\frac{1}{6})$ group where the $2$-cell $R \subset E$ bumps into itself;
	a contradiction~\cite[Lemma 3.10]{Mac-12-random-cdim}.
	
	We now show that the intersection of any two $2$-cells in $X$
	is connected.
	Suppose not.  Then we have $2$-cells $R_1 \ra X, R_2 \ra X$ 
	and non-trivial closed intervals $P_1 \subsetneq \partial R_1, P_2 \subsetneq R_2$
	so that $P_1 \cup P_2$ is an immersed loop in 
	$R_1 \cup_{\partial R_1 \cap \partial R_2} R_2$, where we glue $R_1$ and $R_2$
	along their boundaries' intersection in $X$.
	Consider a disc diagram $D$ for the induced map $P_1 \cup P_2 \ra X$
	of minimal area amongst all such choices of $P_1, P_2, D$.
	Now let $E = R_1 \cup_{P_1} D \cup_{P_2} R_2$ be the disc diagram where
	we glue $R_1$ and $R_2$ to $D$ along $P_1$ and $P_2$, respectively.
	In a similar way as before, $E$ is also a reduced disc diagram.
	In $E$, the two $2$-cells $R_1$ and $R_2$ do not meet along a connected
	interval, which contradicts~\cite[Lemma 3.10]{Mac-12-random-cdim}.
\end{proof}
(The result \cite[Lemma 3.10]{Mac-12-random-cdim} used in the above proof follows
straightforwardly from Lemma~\ref{lem:disc-diag-curvature} below.)

Suppose too that every $1$-cell of $X$ is contained in at least two $2$-cells.
(These assumptions are satisfied if 
every generator of $G$ appears in at least two different places 
in the presentation for $G$.)

We choose $\cV$ to be the set of all V-paths $\alp$ 
so that for all $i = 0, \ldots, l(\alp)$,
the distance between $\alp(i)$ and $\alp(i+1)$ in the boundary
of the $2$-cell $R$ corresponding to $\alp(i+\frac{1}{2})$ is $\geq \frac{3}{8}|R|$.
If a $\cV$-path goes
through $2$-cells $R_1, R_2, R_3$, then $R_1$ and $R_3$ are separated in $\partial R_2$
by at least $(\frac38 - 2\lam)|\partial R_2| > 0$.
Therefore $\cV$ is crossing and reversible.

To apply Theorem~\ref{thm-v-wall-embed} and Proposition~\ref{prop-V-path-qi},
we must check that assumptions~\eqref{eq:as-no-red-V-collared}--\eqref{eq:as-V-path-geod-in-ladder}
are satisfied.
Our main tool is the the following standard fact about disc diagrams.
\begin{lemma}[{\cite[page 241]{Ghys-dlH-90-hyp-groups}}]\label{lem:disc-diag-curvature}
	For any disc diagram $D$ homeomorphic to a disc
	\begin{equation}\label{eq-strebel}
		6 = 2 \sum_v (3-d(v)) + \sum_R (6-2 e(R) -i(R)),
	\end{equation}
	where $d(v)$ is the degree of vertex $v$, and for each $2$-cell $R$,
	$e(R)$ is the number of exterior edges, and $i(R)$ the number of interior edges.
\end{lemma}
This lemma will always be applied, and $i(R), e(R)$ calculated, 
after removing all vertices of degree two from $D$.
For such diagrams the sum over all vertices $v$ contributes $\leq 0$ to \eqref{eq-strebel}.
\begin{remark}\label{rmk:internal-at-least-four}
Suppose we have a disc diagram $D \ra X$ so that there are $2$-cells 
$R_1, R_2, R_3$, with a $\cV$-path $\alp$ going through $R_1, R_2, R_3$.
Then our definition of $\cV$-path implies that if $e(R_2)=1$ then 
$i(R_2) > \frac38 |\partial R_2| / \lam |\partial R_2| \geq 3$, so $i(R_2) \geq 4$.
\end{remark}

\begin{lemma}[Assumption \eqref{eq:as-no-red-V-collared}]\label{lem-no-V-collared}
	There is no reduced, $\cV$-collared diagram.
\end{lemma}
\begin{proof}
	Let $E \ra X$ be such a diagram, with $A \subset E$ the quotient of the corresponding
	ladder, and $\partial E \subset A$.  
	Considered in the plane, $A$ is a connected loop of $2$-cells joined along $1$-cells,
	so $\partial E$ must be embedded and thus $E$ is homeomorphic to a disc.
	
	In $E$, the corner $R$ has $e(R)=1$ and $i(R) \geq 2$,
	so contributes at most $6 -2 \cdot 1 - 2 = 2$ to \eqref{eq-strebel}.
	All interior $2$-cells contribute $\leq 6-9=-3$ to \eqref{eq-strebel}.
	Any non-corner $2$-cell $R$ in the collar is adjacent to at least $4$ other $2$-cells
	by Remark~\ref{rmk:internal-at-least-four}, so contributes 
	$\leq 6-2\cdot 1 -4 = 0$ to \eqref{eq-strebel}.
	These combine to give a contradiction.
\end{proof}

\begin{lemma}[Assumption \eqref{eq:as-no-red-V-triv-collared}]\label{lem:noladderbumpdiagram}
	There is no reduced diagram collared by a $\cV$-path of length $\geq 3$ and a trivial path.
\end{lemma}
\begin{proof}
	Similarly to Lemma~\ref{lem-no-V-collared}, these diagrams are homeomorphic to a disc.
	
	In such a diagram, the contributions to \eqref{eq-strebel} are
	at most $6-2\cdot 1-1 = 3$ for the two corners, 
	at most $6-2\cdot 1 - 4 = 0$ for the other external $2$-cells,
	and at most $6-9 = -3$ for any internal $2$-cells.
	So there are no internal $2$-cells, and the internal edges of $D$ form a tree.
	At a vertex of this tree adjacent to two leaves of the tree 
	the crossing assumption is contradicted.
\end{proof}

Before considering the interaction of $\cV$-paths and geodesics,
we consider geodesics in small cancellation Cayley complexes.
\begin{lemma}\label{lem:sc-twocells-geod-convex}
	Let $X$ be the Cayley complex of a $C'(\lambda)$ group presentation.
	If $\lambda \leq \frac{1}{6}$ then any $2$-cell $R \subset X$ is geodesically
	convex.
	If $\lambda \leq \frac{1}{8}$ then $R \cup R'$ is geodesically convex
	whenever $R, R'$ are $2$-cells in $X$ with $R \cap R' \neq \emptyset$.	
\end{lemma}
\begin{proof}
	First 
	suppose there is a non-trivial geodesic $\gam \subset X^{(1)}$
	which meets $R$ exactly in its endpoints.
	Let $D \ra X$ be a reduced diagram with minimal area so that
	$\partial D = \beta \cup \gam$,
	for some simple path $\beta \subset \partial R$.
	Glue $R$ to $D$ along $\beta$ to find a diagram $E = D \cup R$.
	
	If this diagram is not reduced, 
	there is a cancellable pair $R$ and $R_1$, where $R_1$ is a copy of $R$ in $D$.
	As in the usual reduction of small cancellation theory,
	we can remove the interior of $R \cup R_1$ from $E$ and glue together
	the corresponding pairs of edges remaining in $\partial R \cup \partial R_1$ to 
	to find a diagram of smaller area whose boundary consists 
	of $\gam$ and $\partial R \setminus \beta$.
	This contradicts the choice of $D$, so $E$ must be reduced.
	
	As $\gam$ is a geodesic, it must embed in $\partial E$, so $E$ is homeomorphic to a disc.
	$R$ contributes $\leq 3$ to \eqref{eq-strebel}.
	Any $2$-cell $R'' \subset E$ with $e(R'') \geq 2$ contributes $\leq 6-2\cdot 1 -2 =0$
	to \eqref{eq-strebel}.
	By $C'(\frac{1}{6})$,
	any $2$-cell $R'' \neq R$ with $e(R'')=1$ contributes $\leq 6-2 \cdot 1 -4=0$ to 
	\eqref{eq-strebel}, and any internal $2$-cell contributes $\leq 6-7$ to \eqref{eq-strebel}.
	This contradicts the existence of such $\gam$.
	
	Second, suppose $R \neq R'$, $R \cap R' \neq \emptyset$,
	and that there is a geodesic $\gam$ meeting
	$R \cup R'$ exactly in its endpoints $p \in R \setminus R'$ and $q \in R' \setminus R$.
	Now choose simple paths $\beta \subset R, \beta' \subset R'$
	so that $\beta \cup \beta' \cup \gam$ is a closed loop.
	Let $D \ra X$ be a reduced diagram with $\partial D = \beta \cup \beta' \cup \gam$,
	of minimal area for all such choices of $\beta, \beta'$.
	We glue $R$ and $R'$ to $D$ along $\beta$ and $\beta'$, respectively,
	to find a diagram $E = R \cup D \cup R' \ra X$.
	Similar to above, $E$ is reduced, for if not we could make a reduction which
	would contradict the choice of $D$.
	Moreover, since $\gam$ meets $R \cup R'$ only at its endpoints,
	$E$ is homeomorphic to a disc.
	
	As before, $R, R'$ contribute $\leq 3$ to \eqref{eq-strebel}, 
	$2$-cells $R''$ with $e(R'') = 0$ or $\geq 2$ contribute $\leq 0$,
	but $2$-cells $R''$ with $e(R'')=1$ contribute $\leq 6-2\cdot 1 -5 =-1$.
	Because $\gam$ is non-trivial, there are $2$-cells in $D$ with exactly one exterior edge, 
	so we have a contradiction.	
\end{proof}

\begin{lemma}[Strengthened Assumption \eqref{eq:as-V-path-geod-in-ladder}]\label{lem:geodinladder}
	If $D \ra X$ is a reduced disc diagram collared by a $\cV$-path $\alp$ 
	and a geodesic $\gam \subset X^{(1)}$, then
	$\gam$ is contained in the ladder $L$ of $\alp$ in $D$.
\end{lemma}
\begin{proof}
	Since $L$ consists of $2$-cells glued along $1$-cells, and $\gam$ does not
	self-intersect, $D$ is homeomorphic to a disc.
	
	The (at most) two corners of $L$ contribute at most 
	$6-2\cdot 1 - 1 = 3$ each to \eqref{eq-strebel},
	while any other $2$-cell in $L$ contributes at most $6-2\cdot 1 - 4 = 0$ to 
	\eqref{eq-strebel}.  Internal $2$-cells contribute $<0$.
	
	Every $2$-cell $R$ which has an external edge and which is not in $L$
	must have all its external edges in $\gam$.
	By Lemma~\ref{lem:sc-twocells-geod-convex} we have $e(R)=1$ and
	so $i(R) \geq 5$ and $R$ contributes $6-2-5 \leq -1$ to \eqref{eq-strebel}.
	Therefore there are no such $2$-cells $R$ and $\gam \subset L = D$.
\end{proof}

These lemmas combine with Theorem~\ref{thm-v-wall-embed} and 
Proposition~\ref{prop-V-path-qi} to show that $\cV$-paths are embedded 
and quasi-convex in $X$.
Moreover, by Remark~\ref{rmk:geod-meets-every-cell-in-ladder}
the carriers of $\cV$-paths are geodesically convex.

\subsection{$\cV'$-branching walls in small cancellation groups}\label{ssec-sc-Vprime-def}

We now define the family $\cV' \subset \cV$, by first defining V-paths $\cV'_0 \subset \cV$
of length one.
Fix an orientation of every edge in $X^{(1)}$.
Suppose we have an edge $e$, with midpoint $x$,
contained in the boundary of a $2$-cell $R \subset X$.
We follow $\partial R$ from $e$ in the direction of $e$.
Let $y_1$ be the first edge midpoint we meet so that $d(x,y_1) \geq \frac38 |\partial R|$.
Continuing around $\partial R$, let $y_2$ be the edge midpoint with 
distance $d(y_1,y_2) = \lceil \lambda|\partial R| \rceil-1$. 
Continue, with gaps of $\lceil \lambda |\partial R| \rceil -1$,
defining edge midpoints $y_1, \ldots, y_t$,
until $d(y_{t+1}, x)$ would have been $< \frac38 |\partial R|$.
We add to $\cV'_0$ the V-paths $\alp_i:[0,1] \ra X$ which have $\alp_i(0)=x$, 
$\alp_i(\frac12)$ is the midpoint of $R$, and $\alp_i(1) = y_i$, for $i=1,\ldots,t$.
See Figure~\ref{fig-Vcomplex}.
\begin{figure}
	\centering
	\def\svgwidth{0.7\columnwidth}
	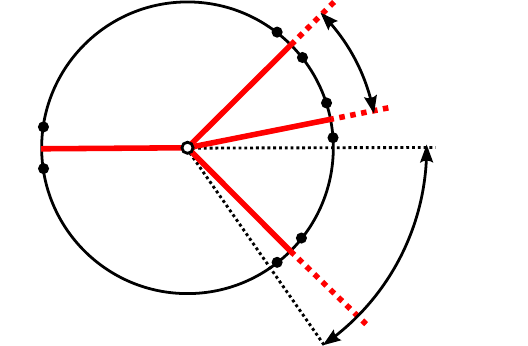
	\caption{Paths in $\cV'_0$}\label{fig-Vcomplex}
\end{figure}

Given $\cV_0'$, we define $\cV' \subset \cV$ to be the set of all V-paths
which restrict to paths in $\cV_0'$ as they cross each $2$-cell.

\begin{lemma}[Assumption \eqref{eq:as-white-vertex-deg-at-least-three}]\label{lem-Vcomplex-valency}
	In the construction of a $\cV'$-complex, we have branching in faces of size
	$t \geq 1+\lfloor 1/8\lambda \rfloor \geq 2$.
\end{lemma}
\begin{proof}
	First note that $\lam |\partial R| >1$ as otherwise the fact that an edge of
	$R$ meets another face would contradict the $C'(\lam)$ assumption.
	In particular, $|\partial R|/8 > 1$, so $2 \lfloor |\partial R|/8 \rfloor \geq |\partial R|/8$.
	Observe that $\lceil \lam |\partial R| \rceil -1 \leq \lam |\partial R|$.
	
	We find $t$ by placing the maximal possible number of objects 
	separated by $\lceil \lam |\partial R| \rceil -1$ 
	into a closed interval of length $2 \lfloor |\partial R|/8 \rfloor$.
	This gives
	\begin{equation*}
		t  = \left\lfloor \frac{2 \lfloor |\partial R|/8 \rfloor}{\lceil \lam |\partial R| \rceil -1} \right\rfloor+1
			\geq \left\lfloor \frac{ |\partial R|/8 }{\lam |\partial R|} \right\rfloor +1 
			= \left\lfloor \frac{1}{8 \lam} \right\rfloor +1 .\ \qedhere
	\end{equation*}
\end{proof}
In our choice of $\cV'$-paths, the $C'(\lambda)$ condition ensures
that when $\cV'$-paths branch, they then go into different $2$-cells.
We formalise this as follows.
\begin{definition}\label{def:Vpath-branching-pair}
	We say that $\cV'$-paths $\alp, \alp'$ of length $\geq 2$ are a \emph{branching pair} if
	$\alp(0)=\alp'(0)$ and $\alp(\frac{1}{2})=\alp'(\frac{1}{2})$, but $\alp(1) \neq \alp'(1)$.
\end{definition}
In our construction, we have:
\begin{equation}\label{eq:as-branch-different-cells}\tag{E}
	\parbox{0.9\textwidth}{
		any branching pair $\alp, \alp'$ satisfies $\alp(\frac{3}{2}) \neq \alp'(\frac{3}{2})$.}
\end{equation}


\section{The complexes decompose $X$}\label{sec-complexes-decompose}

In this section we show that each $\cV'$-complex $Y_e \ra X$ 
decomposes $X$.  Although we work in the situation of a $C'(\lambda)$ group,
we keep track of the assumptions we make on $\cV'$ for possible future applications,
in addition to assumptions~\eqref{eq:as-no-red-V-collared}--\eqref{eq:as-branch-different-cells}
already made.  (See Definition~\ref{def-Vcomplex} for the construction of $Y_e \ra X$.)
\begin{theorem}
	When $X$ is the Cayley complex of a $C'(\lambda)$ group presentation, for some $0 < \lam \leq \frac18$, and $\cV$, $\cV'$ are as defined in 
	Subsections~\ref{ssec-sc-V-def} and \ref{ssec-sc-Vprime-def},
	each $\cV'$-complex $Y_e \ra X$ decomposes $X$ (Definition~\ref{def-decomposes}).
\end{theorem}

The proof of this theorem is contained in the remainder of this section:
the four conditions for $Y_e$ to decompose $X$ are verified by
Lemmas~\ref{lem:Vcplx-decomp-1}, \ref{lem:Vcplx-decomp-2}, \ref{lem:Vcplx-decomp-3} 
and \ref{lem:Vcplx-decomp-4}.

\subsection{Embedding}

We show that $\cV'$-branching walls, and hence $\cV'$-complexes, embed.
\begin{lemma}\label{lem:Vcplx-decomp-1}
	For any edge $e \subset X$, the map $Y_e \ra X$ is an embedding.
\end{lemma}
\begin{proof}
It suffices to show that $Z_e \ra X$ is an embedding.
Suppose we have two $\cV'$-paths $\alp$ and $\alp'$ in $Z_e$.
We want to show that after they diverge in $Z_e$ they never meet again in $X$.
We can assume that $\alp(0)=\alp'(0)$ is the last black vertex that 
$\alp, \alp'$ have in common.

There are two cases, depending on whether $\alp, \alp'$ have a different initial
white vertex or not, i.e.\ whether $\alp(\frac{1}{2}) \neq \alp'(\frac{1}{2})$ or not.
In the first case, $-\alp \cup \alp'$ is an (immersed) $\cV$-path in $X$,
which is an embedding by Theorem~\ref{thm-v-wall-embed}.

In the second case, $\alp$ and $\alp'$ form a branching pair 
(Definition~\ref{def:Vpath-branching-pair}).

Suppose $\alp$ and $\alp'$ meet again after $\alp(\frac{1}{2}) = \alp'(\frac{1}{2})$.
If they first meet in the middle of an edge $e \subset X^{(1)}$,
then the concatenation of $\alp$ up to $e$ and of $-\alp'$ from $e$ back to $\alp'(0)$
gives a self-intersecting $\cV$-path, which contradicts Theorem~\ref{thm-v-wall-embed}.

Suppose that $\alp$ and $\alp'$ meet again in the middle of a $2$-cell;
we assume that this is the last $2$-cell $\alp$ and $\alp'$ cross.

Let $L, L'$ be the ladders of $\alp, \alp'$, and 
as in Definition~\ref{def-quasi-V-collared} let $A \ra X$
be the quotient of $L \cup L'$ on identifying the first and last $2$-cells 
of each $L$; call these quotiented faces $R_1$ and $R_n$.
Because $\alp(\frac{3}{2}) \neq \alp'(\frac{3}{2})$ by assumption~\eqref{eq:as-branch-different-cells}
and since $\alp, \alp'$ are crossing, 
$A$ is homotopic to $\Sph^1$.
Thus we can find a quasi-$\cV$-collared diagram $E \ra X$,
where $E = A \cup_P D$, 
for some disc diagram $D \ra X$, and non-contractible loop $P \subset A$
(Figure~\ref{fig-2Vcollared}).
\begin{figure}
	\centering
	\def\svgwidth{0.6\columnwidth}
	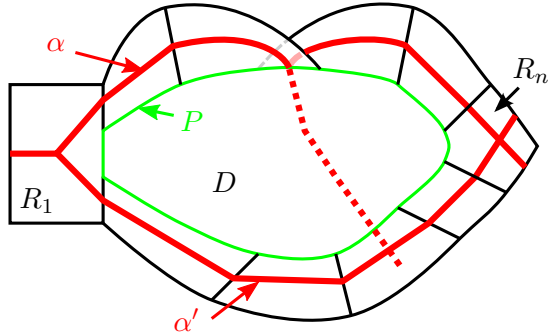
	\caption{A diagram quasi-collared by two $\cV$-paths}\label{fig-2Vcollared}
\end{figure}

We assume that $E$ was chosen to have minimal area.
In a similar way to Theorem~\ref{thm-v-wall-embed},
this means that there are no cancellable pairs 
$R, R' \subset D$ or $R \subset A, R' \subset D$,
as these can be removed leaving $A$ unchanged.
	
There are also no cancellable pairs $R, R' \subset A$:
such a pair would imply that either $\alp$ or $\alp'$ were not embedded, 
or that one could find a diagram of smaller area
quasi-collared by subpaths of $\alp, \alp'$,
both of which are contradictions.

Therefore $E$ is a reduced quasi-$\cV'$-collared diagram. 
As in Lemma~\ref{lem:Vpathembed3}, we can reduce to the case where $P$ does not
meet $\alp \cup \alp'$,
for if the path $P$ crosses, say, $\alp$ then we can extend $\alp$ from this point 
into $D$, and across to $\alp'$ (recall that $\cV$-paths are embedded), 
and extract from that a quasi-$\cV$-collared diagram 
of smaller area 
(see Figure~\ref{fig-2Vcollared}).

We now make an assumption that contradicts the existence of $F$.
\begin{equation}\label{eq:as-no-reduced-almost-2-collared}\tag{F}
	\parbox{0.9\textwidth}{
		Given a branching pair $\alp, \alp'$, there is no reduced 
		diagram collared by $\alp, \alp'$.}
\end{equation}
Given this assumption, we have shown that $\alp, \alp'$ cannot meet again.
So each $Z_e$ is an embedded tree in $X$, and thus $Y_e$ is a connected and simply
connected subcomplex of $X$.
\end{proof}
It remains to check our assumption. 
\begin{lemma}\label{lem:nosplit-bi-collar}
	Assumption~\eqref{eq:as-no-reduced-almost-2-collared} holds.
\end{lemma}
\begin{proof}
	In such a diagram (which is a topological disc), the interior faces contribute
	$\leq -3$ to \eqref{eq-strebel}.
	Exterior faces away from the corners contribute $\leq 6-2\cdot 1 -4 = 0$.
	
	Each of the corners of the diagram has $1$ exterior edge, and $\geq 2$ interior edges,
	so contributes at most $6 -2 \cdot 1 - 2 = 2$ to ~\eqref{eq-strebel}.
	So the total contributed to \eqref{eq-strebel} is $\leq 4$, a contradiction.
\end{proof}

\subsection{Quasi-convexity}

We now show that $\cV'$-branching walls, and hence $\cV'$-complexes, are quasi-convex.
\begin{lemma}\label{lem:Vcplx-decomp-2}
	For any edge $e \subset X$, $Y_e \subset X$ is quasi-convex.
\end{lemma}
\begin{proof}
It suffices to show that $Z_e$ is quasi-convex in $X$.
Suppose that $\alp, \alp'$ are $\cV'$-paths
from $x$, the midpoint of the edge $e$, to $y=\alp(n)$ and $y'=\alp'(n')$, respectively.
We can assume that $x$ is the last black vertex which $\alp, \alp'$ have in common.
We want to show that $d(y,y') \succeq n+n'$ in $X^{(1)}$.

As in the proof of Lemma~\ref{lem:Vcplx-decomp-1} above, we split the proof into two cases.
First, if $\alp(\frac{1}{2}) \neq \alp'(\frac{1}{2})$, then $-\alp' \cup \alp$ is a $\cV$-path in $X$,
which is (uniformly) quasi-convex in $X$ by Proposition~\ref{prop-V-path-qi}.

Otherwise, $\alp, \alp'$ are a branching pair.
Let $L = R_1 \cup R_2 \cup \cdots \cup R_n$ 
and $L' = R_1 \cup R_2' \cup \cdots \cup R_{n'}'$ be the ladders of $\alp$ and $\alp'$.
	
As $\cV'$-complexes embed by Lemma~\ref{lem:Vcplx-decomp-1}, we have $y \neq y'$.
Both $\alp$ and $\alp'$ are quasi-convex by Proposition~\ref{prop-V-path-qi},
and $X$ is Gromov hyperbolic, so
it suffices to show that
$\gam$ comes within a constant distance $C'$ of $x$, for then
there exists $C''$ so that
\begin{align*}
	d(y,y') & \geq \max \{ 1, d(y,x)+d(x,y')-C'' \} \\ &
	\geq \max \{1, n/6+n'/6-C'' \}
	\succeq n+n'.
\end{align*}
	
Consider a diagram $E \ra X$ quasi-$\cV'$-collared by $\alp, \alp'$ and $\gam$.
One can perform the reducing procedure as usual:
since $\alp, \alp'$ are embedded (Theorem~\ref{thm-v-wall-embed}) there are
no cancellable pairs $R_i, R_j$ or $R_i', R_j'$,
and since $\alp$ and $\alp'$ only meet in $R_1$ (Lemma~\ref{lem:Vcplx-decomp-1}),
there are no cancellable pairs $R_i, R_j'$.
So $L, L'$ are preserved, and we are left with a reduced quasi-collared diagram
like Figure~\ref{fig-2Vgeodcollared}.
\begin{figure}
	\centering
	\def\svgwidth{0.65\columnwidth}
	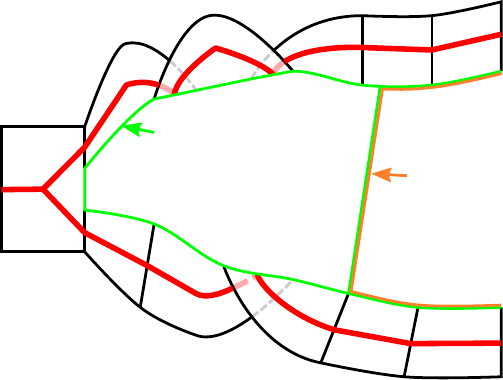
	\caption{}\label{fig-2Vgeodcollared}
\end{figure}

We suppose for a contradiction that $\gam$ does not get within 
four $2$-cells of $R_1$ in $X$.
Consider the subpath of $\gam$ which goes from $L$ to $L'$.
It is enough to consider the case when $\gam$ meets
$A = (L \cup L')/(R_1 \sim R_1')$ only in $R_n \cup R_{n'}'$,
and only at its endpoints.  We have $n, n' \geq 5$.

If in the diagram $E = A \cup_P D \ra X$ the path $P$ crosses $\alp\cup\alp'$,
we can follow (say) $\alp$ into $D$.  
Since $\cV'$-complexes embed by Lemma~\ref{lem:Vcplx-decomp-1},
this $\cV'$-path must leave $D$ by meeting $\gam$.
So we can reduce to the case when $P$ does not cross $\alp \cup \alp'$,
and we still have $\alp, \alp'$ both having length $\geq 5$.

We are left with a reduced, planar diagram which contradicts the following assumption.
\begin{equation}\label{eq:as-no-2-collar-geod}\tag{G}
	\parbox{0.9\textwidth}{
		Given a branching pair $\alp, \alp'$ of lengths $\geq 5$, 
		there is no reduced diagram collared by a $\alp$, a geodesic $\gam$, and $-\alp'$, 
		where $\gam$ meets the ladders of $\alp, \alp'$ only at a single point in each of their final $2$-cells. }
\end{equation}
So the proof is complete.
\end{proof}
It remains to verify our assumption.
\begin{lemma}\label{lem:no-2V-geod-collar}
	Assumption~\eqref{eq:as-no-2-collar-geod} holds.
\end{lemma}
\begin{proof}
	In such a reduced diagram (which is a topological disc),
	any interior $2$-cell contributes $\leq -3$ to \eqref{eq-strebel}.
	Any $2$-cell along the ladders of $\alp, \alp'$ contribute $\leq 0$ to \eqref{eq-strebel}.
	
	We assume that $\gam$ does not meet the penultimate $2$-cells $R_{n-1}, R_{n'-1}'$ of the ladders
	of $\alp,\alp'$,
	so both $i(R_n)$ and $i(R_{n'}')$ are $\geq 2$, and thus they contribute $\leq 6 - 2 - 2 = 2$
	to \eqref{eq-strebel}.
	
	As in the proof of Lemma~\ref{lem:nosplit-bi-collar}, 
	there is a contribution of $\leq 2$ to \eqref{eq-strebel} from $R_1$.
	
	If $\gam$ has non-zero length, then it bounds a face which must have $\geq 5$ internal edges,
	so this face contributes $\leq 6-2-5 = -1$ to \eqref{eq-strebel}.
	
	These bounds combine to show that $\gam$ must be a trivial path and there are no internal faces.
	Therefore the internal edges of the diagram form a tree.
	However, both $i(R_n)=i(R_{n'}')=2$, and $i(R_{n-1})=i(R_{n'-1}')=4$; this gives a contradiction.
\end{proof}

\subsection{Connected components of the complement}

The connected components of $X \setminus Y_e$ can be described explicitly;
since $Y_e$ is contractible and $X$ is simply connected, we can see the connected
components of $X \setminus Y_e$ by considering $Y_e \setminus Z_e$.

We orient $Y_e$ by choosing an orientation of the edge $e$, and an orientation on the
$2$-cells adjacent to $e$ so that $e$ goes around their boundaries clockwise.
There are two connected components $H_l$ and $H_r$ adjacent to $e$.
The first is bounded by the $\cV'$-wall found by extending the wall
along the left-most option every time, and the second is bounded by 
the wall which extends along the right-most option every time; see Figure~\ref{fig-Vtree}.

The other connected components are found by following a $\cV'$-path out to a face
with white vertex $q$, then branching at $q$ to two adjacent $\cV'$-paths through black
vertices $v_l, v_r$ as in Figure~\ref{fig-Vtree};
the component $H$ is bounded by the $\cV'$-walls found by 
taking all the right-most choices following $v_l$, 
and all the left-most choices following $v_r$.

Since $Z_e$ is quasi-convex, $H$ is also quasi-convex.
Because $X$ is Gromov hyperbolic, every geodesic from $e$ to $H$ 
must pass within a bounded distance of $q$.

\begin{lemma}\label{lem:Vcplx-decomp-3}
Each pair $H_1, H_2$ of distinct connected components of $X \setminus Y_e$ has 
$\bdry H_1 \cap \bdry H_2 = \emptyset$.
\end{lemma}
\begin{proof}
Suppose $z \in \bdry H_1 \cap \bdry H_2$,
and let $\gam$ be a geodesic ray from $e$ to $z$.
	
By the thin triangles condition, and quasi-convexity of $H_1$,
$\gam$ eventually lies in a $C$-neighbourhood of $H_1$, for some $C>0$.
It also lies in a $C$-neighbourhood of $H_2$.
This contradicts the quasi-convex embedding of $Y_e$,
as far away from $e$ the distance between $H_1$ and $H_2$ should grow
arbitrarily large.  
\end{proof}

\begin{lemma}\label{lem:Vcplx-decomp-4}
	Every sequence $\{ H_i\}$ of distinct connected components of $X \setminus Y_e$ 
	subconverges in $X \cup \bdry X$ to a point in $\bdry X$.
\end{lemma}
\begin{proof}
	Choose a subsequence which does not contain $H_l$ and $H_r$,
	and let the associated branch points be denoted by $q_i$.
	Necessarily $d(e, q_i) \ra \infty$, so we must have $d(e, H_i) \ra \infty$
	also.
	Since any geodesic from $e$ to $\bdry H_i$ has to go through the ball $B(q_i,C)$,
	we have $\diam(\bdry H_i) \preceq e^{-\eps d(e,q_i)} \ra 0$, where $\eps$ is the
	visual parameter for the visual metric on the boundary.
	Finally, choose a subsequence so that $\bdry H_i$ converges to a point
	in $\bdry X$.
\end{proof}

\subsection{A general statement}
We observe that we have actually shown the following, where we need not be working
with the Cayley complex of a $C'(\frac18)$ small cancellation group.
\begin{theorem}\label{thm-Vcomplex-decompose}
	Suppose $X$ is as in Assumption~\ref{assump-main},
	and that $\cV' \subset \cV$ are crossing families of V-paths, and $\cV$ is reversible.
	If assumptions \eqref{eq:as-no-red-V-collared}--\eqref{eq:as-no-2-collar-geod} are satisfied,
	then each $\cV'$-complex $Y_e \ra X$ decomposes $X$.
\end{theorem}


\section{Verifying the upper bound}\label{sec-new-upper}

In this section, we use the machinery developed to prove Theorem~\ref{thm-upper-bound},
and apply this result to random groups.

\subsection{The complexes fully decompose $X$}
It remains to show that we can separate points in $\bdry X$ by $\cV'$-complexes.

\begin{lemma}\label{lem:fullydecompose}
	For every $z_1 \neq z_2$ in $\bdry X$, there exists $Y_e \ra X$ so that
	for every connected component $E$ of $\overline{X \setminus Y_e}$,
	$\{ z_1, z_2 \} \nsubseteq \bdry E$.
\end{lemma}	
\begin{proof}
	Let $\gam$ be a geodesic from $z_2$ to $z_1$.
	Consider the overlap of $\gam$ with faces in $X$.
	If $\gam$ does not contain a subpath of length 
	$\frac{1}{4}|\partial R|$ of any $2$-cell, let $e$
	be any edge of $\gam$; this is Case (i).
	If $\gam$ does contain a subpath of length
	$\frac{1}{4}|\partial R|$ of a $2$-cell $R$, let $e$ be the
	edge in the middle of this subpath; this is Case (ii).
	
	In either case, let $H_l, H_r$ be the two connected components of $X \setminus Y_e$
	which $e$ meets.
	It suffices to show that $z_1,z_2$ are in $\bdry H_l, \bdry H_r$ respectively,
	as $\bdry H_l \cap \bdry H_r = \emptyset$.
	By Proposition~\ref{prop-V-path-qi}, it suffices to show that the ray of $\gam$ going
	to $z_1$ leaves the carrier supporting $Y_e$ into $H_l$, for then the ray must
	always remain in $H_l$ (and likewise for the ray to $z_2$).
	
	\emph{Case (i)}: If $\gam$ stays in the carrier of $Y_e$, it must follow $Y_e$ from
	$e$ along the boundary of a $2$-cell $R_1$ into an adjacent face, 
	see Figure~\ref{fig-leaveladder}, ignoring $R$.
	By the definition of a $\cV$-path, and $C'(\frac{1}{8})$, this forces $\gam$ to contain
	at least $\frac14 |\partial R_1|$ of $R_1$, contradicting Case (i).
\begin{figure}
	\centering
	\def\svgwidth{0.7\columnwidth}
	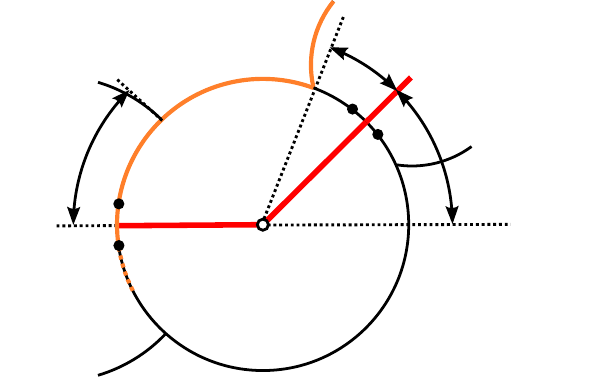
	\caption{$Y_e$ separates $z_1$ and $z_2$.}\label{fig-leaveladder}
\end{figure}
	
	\emph{Case (ii)}: First, if $\gam$ uses $R$ to follow
	$Y_e$ into an adjacent face, by case (i) this subpath of $\gam$ must
	contain $>\frac14 |\partial R|$ of $\partial R$.
	But $e$ was chosen to be the midpoint of $\gam \cap R$, so
	$\gam$ must meet $R$ along a path of length $>\frac12 |\partial R|$,
	which contradicts $\gam$ being geodesic.
	
	Second, $\gam$ cannot use a relation $R_1 \neq R$ to follow $Y_e$.
	Figure~\ref{fig-leaveladder} shows that this forces $\gam$ to contain
	at least $\frac14 |\partial R_1|$ of $\partial R_1$ as it travels from $e$, 
	as well as $\frac18 |\partial R|$ of $\partial R$.
	This contradicts $C'(\frac{1}{8})$.	
\end{proof}

\subsection{Proof of upper bound for small cancellation groups}
\begin{proof}[Proof of Theorem~\ref{thm-upper-bound}]
	Without loss of generality, all generators in $\cS$ appear in at least two relators.
	For if any generator appears in only one relator, we can discard the generator and relator.
	If a generator $s$ appears in no relator, $G$ is the free product of $\langle s \rangle$
	and the group $H$ generated by the other generators.
	This contradicts the assumption that $G$ is one-ended.
	
	By Lemmas~\ref{lem:Vcplx-decomp-1}, \ref{lem:Vcplx-decomp-2}, \ref{lem:Vcplx-decomp-3},
	\ref{lem:Vcplx-decomp-4} and~\ref{lem:fullydecompose}, the $\cV'$-complexes $\{Y_e \ra X\}$ fully decompose
	the Cayley complex $X$.
	We now apply Theorem~\ref{thm-bourdonkleiner},
	where the thickness of the black edges is in the interval $[2,k]$, 
	and the perimeter is at least $2(t+1)$.  Using $t \geq \lfloor 1/8\lam \rfloor +1$ 
	(Lemma~\ref{lem-Vcomplex-valency}),
	we see that
	\[
		\Cdim(\bdry G) \leq 1+ \frac{\log(k-1)}{\log(\lfloor 1/8\lam \rfloor +1)}. \qedhere
	\]
\end{proof}

\subsection{Applications to random groups}

For fixed $m \geq 2$ and $n \geq 1$, a random $m$ generator, $n$ relator group with relators
of length $\leq l$ has, a.a.s.,
a $C'(\lam)$ presentation for $\lam = 11\log(l)/(l \log(2m-1))$ 
\cite[Proposition 2.2]{Mac-12-random-cdim}.
We now show an analogous bound for polynomial density random groups.
\begin{proposition}\label{prop:poly-rand-small-canc}
	Suppose $C > 0$ and $K \in [0, \infty)$ are fixed,
	and $G = \langle \cS | \cR \rangle$
	is a random $m$ generator, 
	$n= Cl^K$ relator group, then a.a.s.\ 
	$G$ has a $C'(\lam)$ presentation, for $\lam = 6(K+2)\log(l)/ (l\log(2m-1))$.
\end{proposition}
\begin{proof}
	Since the number $N_l$ of cyclically reduced words of length $l$ grows exponentially,
	and $n$ grows polynomially, a.a.s.\ every relator has length at least $0.99l$
	(see the discussion before the proof of \cite[Proposition 2.2]{Mac-12-random-cdim}).
	
	First we bound the probability $P_1$ that two different relators
	share a word $u$ of length at least $|u| \geq \lam 0.99l$.
	We have $\binom{n}{2} \leq n^2$ choices for the relators $r_i, r_j \in \cR$,
	and $(2l)^2 = 4l^2$ choices for the starting position of $u$ in $r_i^{\pm 1}, r_j^{\pm 1}$.
	The probability that the subword of $r_j^{\pm 1}$ then matches the corresponding subword
	of $r_i$ is $\preceq (2m-1)^{-0.99 \lam l}$.
	Combined, we have
	\begin{align*}
		P_1 & \preceq n^2 (2l)^2 (2m-1)^{-0.99 \lam l}
		\leq 4C^2 l^{2K+2} l^{-0.99 \cdot (6K+12)} \ra 0.
	\end{align*}
	
	Second, we bound the probability $P_2$ that a word $u$ of length 
	$|u| \geq \lam 0.99l$ appears as a subword of a relator $r \in \cR$ in two different ways.
	By \cite[Lemma 2.3]{Mac-12-random-cdim}, if this occurs then there is word $v$ of 
	length at least $\lam l /5$ which appears in $r^{\pm 1}$ in two different, non-overlapping 
	locations.
	The probability of this occurring is bounded by the product of the 
	number $n$ of choices of $r$, 
	the number $(2l)^2$ of locations for the copies of $v$,
	and the probability that the subword starting at the second location matches the 
	subword at the first, which is $\preceq (2m-1)^{-\lam l/5} $.
	So
	\begin{align*}
		P_2 & \preceq n (2l)^2 (2m-1)^{-\lam l/5}
		= 4C l^{K+2} l^{-6(K+2)/5} \ra 0.
	\end{align*}
	
	Since the probability that $G$ fails $C'(\lam)$ is bounded by $P_1+P_2$,
	we are done.
\end{proof}

We now show the upper bounds of Theorems~\ref{thm-main-polygrowth} and \ref{thm-main-density}.

\begin{corollary}\label{cor:randpolyupper}
	Suppose $C > 0$ and $K \in [0, \infty)$ are fixed,
	and $G = \langle \cS | \cR \rangle$ 
	is a random $m$ generator, $n= Cl^K $ relator group, then a.a.s.\
	\[
		\Cdim(\bdry G) \leq 2+K+\frac{2(K+1) \log\log l}{\log l}\ . 
	\]
\end{corollary}
\begin{proof}
	Every reduced word in $S$ of length $12$ appears, a.a.s., as a subword of some $r \in \cR$;
	in fact, every such word will appear in the first relator.
	By Proposition~\ref{prop:poly-rand-small-canc} such a presentation 
	is $C'(\frac{1}{12})$, a.a.s., so \cite[Theorem 4.18]{Cha-95-rand-grps}
	shows that $G$ is one-ended and that $\bdry G$ is homeomorphic to the Menger curve.
	
	The conformal dimension bound now follows from 
	Theorem~\ref{thm-upper-bound} applied to $G$ with 
	$M=l$, $|\cR|=n = \lceil C l^K \rceil$ and $\lam=6(K+2)\log(l)/ (l\log(2m-1))$.
	Observe that
	$\log(\lfloor 1/8\lam \rfloor + 1) = \log l - \log\log l +O(1)$, 
	and $\log(|\cR|M) = (K+1)\log l +O(1)$,
	so, a.a.s.,
	\begin{align*}
		\Cdim(\bdry G) & \leq 1+ \frac{(K+1)\log l+O(1)}{\log l -\log\log l+O(1)} \\
		& \leq 1+(K+1)+\frac{(K+1)\log\log l +O(1)}{\log l -\log\log l +O(1)} \\
		& \leq 2+K+\frac{2(K+1) \log\log l}{\log l}.\qedhere
	\end{align*}
\end{proof}

The case of the density model is even simpler.

\begin{corollary}\label{cor:densityupper}
	There exists $C > 0$ so that for any density $d < \frac12$, a.a.s.\ 
	a random $m$-generated group $G$ at density $d$ has
	\[
		\Cdim(\bdry G) \leq C \log(2m-1) \left(\frac{d}{|\log d|} \vee \frac{1}{1-2d} \right) l.
	\]
\end{corollary}
\begin{proof}
	For densities $d < \frac12$ we have the fairly direct upper bound
	\begin{equation}\label{eq:upperdensity-highd}
		\Cdim(\bdry G) \preceq \log(2m-1) \cdot \frac{l}{1-2d},
	\end{equation}
	by \cite[Proposition 1.7]{Mac-12-random-cdim}.  (This follows from
	estimating the Hausdorff dimension of $\bdry G$ for a visual metric.)
	
	When $2d < \lam \leq \frac{1}{8}$, then a.a.s.\ $G$ has a $C'(\lam)$ presentation
	\cite[Section 9.B]{Gro-91-asymp-inv},
	and is one-ended~\cite{DGP-10-density-menger}.
	We have $\log(\lfloor 1/8\lam \rfloor +1) \succeq |\log d\,|$,
	and so for $d < \frac{1}{16}$, Theorem~\ref{thm-upper-bound} gives
	\begin{equation}\label{eq:upperdensity-lowd}
		\Cdim(\bdry G) = 1 + \frac{\log\big( (2m-1)^{dl} l \big)}{\log(\lfloor 1/8\lam \rfloor +1)}
		\preceq \log(2m-1) \cdot \frac{dl}{|\log d\, |}.
	\end{equation}
	
	The corollary follows from \eqref{eq:upperdensity-highd} and \eqref{eq:upperdensity-lowd}.	
\end{proof}


\section{A lower bound for random few relator groups}\label{sec-new-lower-sc}

The conformal dimension of a metric space can be bounded from below by finding
within the space a product of a Cantor set and an interval.
One way to build such a set in the boundary of a hyperbolic space is to find a 
`round tree' inside the space itself.
This was done for certain small cancellation groups in~\cite[Sections 5 and 6]{Mac-12-random-cdim}.

In this section we find sharp lower bounds on the conformal dimension of a random group
when the number of relators is constant, or growing polynomially fast.
We do this by building a bigger round tree in the Cayley complex of
such a group, extending methods from \cite{Mac-12-random-cdim}.

We begin by summarising work from \cite[Section 6]{Mac-12-random-cdim}.
Recall that for complexes $A' \subset A$, the \emph{star} $\St(A')$ of $A'$ (in $A$)
is the union of all closed cells which meet $A'$.
\begin{definition}\label{def:comb-round-tree}
	We say a polygonal 2-complex $A$ is a \emph{combinatorial round tree} 
	with vertical branching $V \in \N$ 
	and horizontal branching at most $H \in \N$ if, setting $T = \{1,2,\ldots,V\}$,
	we can write
	\[
		A = \bigcup_{\ba \in T^\N} A_\ba, 
	\]
	where
	\begin{enumerate}	
	\item $A$ has a base point $1$, contained in the boundary of a unique 
	2-cell $A_\emptyset \subset A$.  
	
	\item Each $A_\ba$ is an infinite planar 2-complex, homeomorphic to
	a half-plane whose boundary is the union of two rays $L_\ba$ and $R_\ba$ with 
	$L_\ba \cap R_\ba = \{1\}$.
	
	\item Set $A_0 = A_\emptyset$, and for $n > 0$, let $A_n = \St(A_{n-1})$.
	Given $\ba = (a_1, a_2, \ldots) \in T^\N$, let $\ba_n = (a_1, \ldots, a_n) \in T^n$.
	If $\ba, \bbb \in T^\N$ satisfy $\ba_n = \bbb_n$ and $\ba_{n+1} \neq \bbb_{n+1}$, then
	\[
		A_n \cap A_\ba \subset A_\ba \cap A_\bbb \subset A_{n+1} \cap A_\ba.
	\]
	We require that each 2-cell $R \subset A_n$ meets at most 
	$VH$ 2-cells in $A_{n+1} \setminus A_n$.
	\end{enumerate}
\end{definition}
The picture to have in mind is that each $A_n$ is a union of $V^n$ different
planar $2$-complexes $\{A_{\ba_n}\}$ indexed by $\ba_n \in T^n$.
Each $A_{\ba_n}$ is homeomorphic to a disc, and its boundary path
consists of $L_{\ba}\cap A_n, R_\ba \cap A_n$ and a connected path $E_{\ba_n}$.
We build $A_{\ba_{n+1}}$ from $A_{\ba_n}$ by attaching $2$-cells along $E_{\ba_n}$
so that each $2$-cell in $A_{\ba_n}$ is adjacent to at most $H$ new $2$-cells.
(See \cite[Figure 4]{Mac-12-random-cdim}.)

\begin{theorem}\label{thm-roundtree-cdim}
	Let $X$ be a hyperbolic polygonal 2-complex.
	Suppose there is a 	combinatorial round tree $A$ with vertical branching $V \geq 2$
	and horizontal branching $H \geq 2$.
	Suppose further that $A^{(1)}$, with the natural length metric giving each edge length one,
	admits a quasi-isometric embedding into $X$.
	Then
	\[
		\Cdim(\bdry X) \geq 1+ \frac{\log V}{\log H}.
	\]
\end{theorem}
\begin{proof}
	This is proved in \cite[Section 6]{Mac-12-random-cdim}, using slightly different terminology,
	where the final equation of \cite[page 237]{Mac-12-random-cdim} states that
	\begin{equation}\label{eq:lower-bound-reference}
		\Cdim(\bdry A) \geq 1+\frac{\sigma}{\sigma-\tau} = 1+ \frac{\log |T|}{\log M}.
	\end{equation}
	(In \cite{Mac-12-random-cdim}, `$X$' is a space quasi-isometric to $A$.)
	In our case we have $|T| = V$.
	Since each face in $A_n$ meets at most $VH$ faces in $A_{n+1}\setminus A_n$,
	we replace the definition of $W$ in \cite{Mac-12-random-cdim} by $W = \{1,2,\ldots, VH\}^\N$.
	This results in replacing $M$ by $H$ in \eqref{eq:lower-bound-reference}.
\end{proof}

\subsection{Short subwords in the polynomial density model}

We begin by bounding the probability that a random cyclically reduced
word omits a set of prescribed subwords.
\begin{lemma}\label{lem:omitmanyshortwords}
	Fix $j$ different reduced words of length $g(l)<l/4$ in $\langle \cS \rangle$, where
	$|\cS| = m \geq 2$ and $l \geq 4$.
	The probability that a random cyclically reduced word of length $l$ in $\langle \cS \rangle$ omits
	all $j$ words is at most
	\[
		\exp\left( \frac{2}{(2m-1)^{(l/2)-1}} - \frac{l j}{9g(l)(2m-1)^{g(l)}} \right).
	\]
\end{lemma}
\begin{proof}
	This largely follows \cite[Lemma 2.5]{Mac-12-random-cdim}.
	We split a reduced word $r$ into an initial letter, then words $u_1, \ldots, u_A$ of
	length $g(l)+1$, then a tail of length between $(l/2)-1$ and $3l/4$.
	We can take $A = \lfloor l/(2g(l)+2) \rfloor$.
	Each $u_i$ consists of an initial letter, then a word of length $g(l)$,
	which is forbidden to be any of the $j$ specified words.
	
	In our case, for each $s \in \cS^{\pm}$, let $p_s$ be the number of forbidden words
	which begin with $s$, so $\sum_{\cS^{\pm}} p_s = j$.
	If the initial letter of $u_i$ is $s^{-1}$, we have $(2m-1)^{g(l)}-j+p_s$
	choices for the remainder of $u_i$.
	One choice $s_*$ for the initial letter of $u_i$ is forbidden by the previous letter,
	so the total number of choices for the word $u_i$ is 
	\begin{align*}
		& \sum_{s \in S^{\pm}\setminus \{s_*\}}
			\big( (2m-1)^{g(l)}-j+p_s \big)
		\\ & = (2m-1)^{g(l)+1}-(2m-1)j+ \bigg( \sum_{s \in S^{\pm}}p_s \bigg) - p_{s_*}
		\\ & \leq (2m-1)^{g(l)+1}-(2m-2) \cdot j.
	\end{align*}
	In the proof of \cite[Lemma 2.5]{Mac-12-random-cdim} for $j=1$, the number
	of choices for $u_i$ was bounded by $(2m-1)^{g(l)+1}-(2m-2) \cdot 1$.
	On replacing $1$ by $j$, the remainder of that proof gives our lemma.
\end{proof}

As a consequence, we show that, a.a.s., every subword of a certain length appears as
a subword in a random group presentation.
\begin{proposition}\label{prop:everyword-poly}
	Suppose $C>0$ and $K \in [0,\infty)$, and $G = \langle \cS | \cR \rangle$ is a
	random $m$ generator, $n= Cl^K $ relator group, then 
	a.a.s.\ every word in $\cS$ of length 
	\[
		t =  \big( (K+1)\log l - 3 \log\log l \big) / \log(2m-1)
	\]
	appears as a subword of some relator.	
\end{proposition}
\begin{proof}
	Let $l_1, \ldots, l_n$ be the lengths of the relators; 
	we can assume that $0.99l \leq l_i \leq l$ for $i=1,\ldots,n$.
	
	By Lemma~\ref{lem:omitmanyshortwords}, 
	the probability that a word of length 
	$l_i \in [0.99l,l]$
	omits a given word of length $t$ is at most
	\[
		P_1 = \exp \left( \frac{2}{(2m-1)^{l/4}} - \frac{l}{9t(2m-1)^t} \right),
	\]
	so the probability that all $n$ relators miss a given word of length $t$ is at most
	$P_1^n$.
	Therefore the probability that one of the $2m(2m-1)^{t-1}$ words of length $t$
	is missed by all $n \asymp l^K$ relators is bounded by
	\begin{align*}
		2m(2m-1)^{t-1} P_1^n 
		& \preceq \exp \left( t\log(2m-1)+\frac{2n}{(2m-1)^{l/4}} - \frac{ln}{9t(2m-1)^t} \right)\\
		& \preceq \exp \left( (K+1)\log l 
			- \frac{l Cl^K \log(2m-1)}{9(K+1)(\log l) 
			l^{K+1}/\log^{3} l}\right),
	\end{align*}
	which goes to zero as $l\ra\infty$.
\end{proof}

\subsection{Perfect matchings and lower bounds}

Consider the following problem: given $a \in G$ 
and a reduced word $u$,
when do we have $d(1,au) = d(1,a)+d(a,au)$?
In $C'(\frac16)$ groups, when $|u|$ is less than $\frac{1}{6}$ of the length of
the shortest relator,
this is ensured by ruling out two initial letters for $u$.
\begin{definition}\label{def:nonextendingneighbours}
	A point $a$ in the Cayley graph $\Gam = X^{(1)}$ of a finitely generated group $G$
	has \emph{$i$ non-extending neighbours} if there are distinct $b_1, \ldots, b_i \in G$
	so that $d(b_j,a)=1$ and $d(1,b_j) \leq d(1,a)$, for $j=1,\ldots,i$.
\end{definition}
Observe that, as $d$ is a word metric, every point $a \neq 1$
has at least one non-extending neighbour.
\begin{lemma}[{cf.\ \cite[Lemma 5.3]{Mac-12-random-cdim}}]\label{lem:geod-extend-sc}
	Suppose $G=\langle \cS | \cR \rangle$ is a $C'(\frac{1}{6})$ presentation
	with all relators of length at least $M'$, and $\eta\in \Z_{\geq 0}$ satisfies
	$\eta+1<M'/6$.
	Given $a \in \Gam$, let $E_{a, \eta}$ be the collection of points
	$b \in \Gam$ so that $d(1,b) = d(1,a)+d(a,b)$, $d(a,b) \leq \eta$,
	and $b$ has at least two non-extending neighbours.
	
	Then for every point $a \in G$, 
	either $E_{a, \eta} = \emptyset$, or $E_{a, \eta} = \{b\}$ where $b$ has
	exactly two non-extending neighbours.
\end{lemma}
\begin{proof}
	The proof is identical to that of \cite[Lemma 5.3]{Mac-12-random-cdim}.
	Suppose that $E_{a,\eta}$ contains $b_1, b_2$, with a fixed geodesic $[a,b_i]$ for $i=1,2$,
	and that for $i=1,2$ we have $c_i \in G$ with $c_i \notin [a,b_i]$, $d(b_i, c_i)=1$,
	and $d(1,c_i) \leq d(1,b_i)$.
	
	Consider the geodesic triangle with sides $[1,b_i],[b_i,c_i],[1,c_i]$, where
	$[1,b_i]=[1,a]\cup [a,b_i]$.  
	A reduced diagram for this triangle is a ladder \cite[Lemma 3.12]{Mac-12-random-cdim} and
	contains a $2$-cell $R_i$ which has $b_i,c_i \in \partial R_i$.
	Since $[1,c_i]$ is a geodesic, $\partial R_i$ meets $[1,a]$ along a segment of
	length
	\[
		\geq \tfrac12 |\partial R_i| - \tfrac16 |\partial R_i|-d(a,b_i)-d(b_i,c_i) 
		\geq \tfrac13 |\partial R_i|- \eta-1 \geq \tfrac16 |R_i|.
	\]
	So the $C'(\frac{1}{6})$ condition implies that $R_1=R_2$ in $X$.
	As we travel along $\partial R_i$ from $a$ to $b_i$, the distance to $1$ increases until
	we reach $c_i$.  Since $R_1=R_2$, we then have $b_1=b_2$ and $c_1=c_2$ as required.
\end{proof}

When building $A_{n+1}$, we split $E=E_{\ba_n}$ into segments of length 
between $3$ and $6$.
At each segment endpoint, we find a path of length three extending away from $1$
avoiding the direction ruled out by \cite[Lemma 5.4]{Mac-12-random-cdim},
and the direction ruled out (if any) by Lemma~\ref{lem:geod-extend-sc}.
Therefore, from this endpoint we can further extend along any of $(2m-1)^{\eta-3}$ possible
paths and we will be building geodesics in $X$, see Figure~\ref{fig-extend-round-tree}.
These geodesics have distinct endpoints by small cancellation.
\begin{figure}
	\centering
	\def\svgwidth{0.55\columnwidth}
	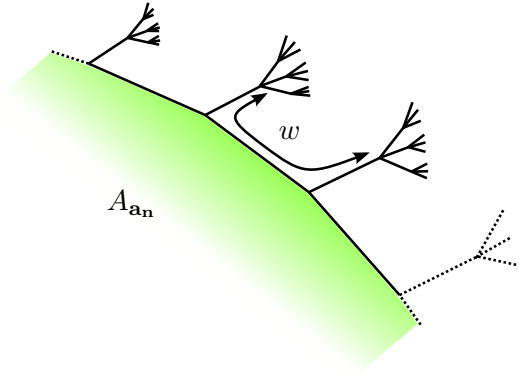
	\caption{Extending round trees}\label{fig-extend-round-tree}
\end{figure}

The aim is to match up the endpoints of these trees with $2$-cells, so as to build
an embedded round tree with branch set indexed by $T$ of size $(2m-1)^{\eta-3}$.
For each reduced word $w \in \langle \cS \rangle$ of length $\in [9,12]$,
consider the labelled tree $\Gamma_w$ found by attaching $(2m-1)$-valency rooted trees to the left
and to the right of a line of length $|w|$ labelled by $w$.  
Orient and label the edges in the trees
so that every vertex outside the interior of $w$ has degree $2|\cS|$ 
and has one ingoing and one outgoing edge for each generator.
Let $H_w$ denote the $(|T|,|T|)$ bipartite graph with vertices corresponding to the
endpoints in $\Gamma_w$, and add an edge joining $v_1, v_2$ whenever the corresponding labelled 
simple path in $\Gamma_w$ of length $2\eta-6+|w|$ can be found in some relator in $\cR$.
What we want is to find a perfect matching in $H_w$.
\begin{proposition}\label{prop:poly-matching}
	Suppose $C>0$ and $K \in [0,\infty)$, and $G = \langle \cS | \cR \rangle$ is a
	random $m$ generator, $n= Cl^K$ relator group, then 
	a.a.s.\ 
	for 
	\[
		\eta =  \big( (K+1)\log l - 4 \log\log l \big) / \log(2m-1),
	\]
	for every word $w$ with $9 \leq |w| \leq 12$, $H_w$ contains a perfect matching.
\end{proposition}
\begin{proof}
	It suffices to check a single word $w$.
	The following is inspired by a result on random bipartite 
	graphs~\cite[Theorem~4.1]{JLR-00-Random-graphs}.
	
	Observe that by Proposition~\ref{prop:everyword-poly},
	every word of length $\leq \eta+9$ appears as a subword somewhere in $\cR$,
	so there is no isolated vertex in $H_w$.
	
	Suppose that $H_w$ does not contain a perfect matching.
	Then by Hall's marriage theorem, there exists $W$ contained either in the left or right vertices
	so that $|W| > |N(W)|$, where $N(W)$ denotes the neighbours of $W$.
	Choose $W$ of minimal cardinality; such a $W$ must satisfy
	$|W|=|N(W)|+1$ (else discard a vertex of $W$) 
	and $|W| \leq \lceil |T|/2 \rceil$ (else replace $W$ by $N(W)^c$, where $N(W)^c$ denotes
	the complement of $N(W)$ in its half of the vertices).
	
	Given a choice of $W$ and $N(W)$ with $|W|=s \geq 2$ and $|N(W)|=s-1$,
	none of the $s(|T|-s+1)$ words of length $l'=2\eta-6+|w|$ from $W$ to $N(W)^c$ can appear as subwords in $\cR$.
	The probability that all $C l^K $ relators omit these words is bounded,
	using Lemma~\ref{lem:omitmanyshortwords}, by
	\begin{align*}
		P_1 
		& = \exp \left( \frac{2}{(2m-1)^{l/4}} - \frac{ls(|T|-s+1)}{9l'(2m-1)^{l'}} 
			\right)^{C l^K } \\
		& \leq 2 \exp \left( - \frac{Cl^{K+1}s(|T|-s+1)}{9l'(2m-1)^{l'}} \right) \\
		& \leq 2 \exp \left( - \frac{Cl^{K+1}s|T|}{18l'(2m-1)^{l'}} \right)
	\end{align*}
	for large $l$, where we use that $s \leq \frac12 |T|+1$.
	
	There are $\binom{|T|}{s} \leq |T|^s$ choices for $W$ and 
	$\binom{|T|}{s-1} \leq |T|^{s-1}$ choices for $N(W)^c$,
	so the probability that we have no perfect matching is:
	\begin{align*}
	& \leq \sum_{s=2}^{\lceil |T|/2 \rceil}
		|T|^s \cdot |T|^{s-1}
		\cdot 2 \exp \left( - \frac{Cl^{K+1}s|T|}{18l'(2m-1)^{l'}} \right)
	\\ & \leq 2\sum_{s=2}^{\lceil |T|/2 \rceil}
		\exp \left( (2s-1)\log |T|
		- \frac{Cl^{K+1}s(2m-1)^{\eta-3}}{18(2\eta+6)(2m-1)^{2\eta+6}} \right)
	\\ & \leq 2\sum_{s=2}^{\lceil |T|/2 \rceil}
		\exp \left( 2s(K+1)\log l-
		 \frac{C'l^{K+1}s}{(\log l) l^{K+1}/\log^4l} \right),
	\end{align*}
	for some $C'=C'(m,K,C)$.
	For large $l$ this is a geometric series with ratio $\leq \frac{1}{2}$,
	and arbitrarily small initial term, hence this bound goes to zero as $l \ra \infty$.
\end{proof}

Continuing by induction, we build a polygonal complex $A = \bigcup_n A_n$
with an immersion of $A$ into the Cayley complex of $G$.
By \cite[Section 5.2.2, 5.2.3]{Mac-10-confdim}, this complex is quasi-isometrically
embedded.
It has vertical branching $V$, where
\[
	\log V = (\eta-3)\log(2m-1) \geq (K+1)\log l - 5 \log\log l,
\]
and horizontal branching $H \leq l$.
Thus by Theorem~\ref{thm-roundtree-cdim}, for large $l$ we conclude that 
\[
	\Cdim(\bdry G) \geq 1 + \frac{\log V}{\log H} \geq 2+K - \frac{5 \log \log l}{\log l}.
\]
This proves our desired lower bound, and completes the proof of Theorem~\ref{thm-main-polygrowth}.\qed


\section{Lower bounds at higher density}\label{sec-lower-density}

In this section, we show the existence of suitable round trees in the Cayley complexes
of random groups at densities $d<\frac{1}{8}$, which give asymptotically sharp bounds on the conformal 
dimension of these groups.
The main problem is that at densities $d<\frac{1}{8}$ we can only expect 
the very weak $C'(\frac{1}{4})$ condition.  
To overcome this difficulty
we use Ollivier's isoperimetric inequality,
and many modifications to the argument of \cite{Mac-12-random-cdim}.

The improved lower bound (of order $dl/|\log d|$ rather than order $dl/\log l$)
results from attaching $2$-cells in the round tree along paths longer than
six.  This idea can also be used to improve \cite[Theorem 5.1]{Mac-12-random-cdim},
which we record for completeness.
\begin{theorem}\label{thm-lower-bound-sc}
	Suppose $G = \langle \cS | \cR \rangle$ is a $C'(\frac{1}{8} -\delta)$ presentation, 
	with $|\cS| = m \geq 2$ and $|\cR| \geq 1$, where $\delta \in (0, \frac{1}{8})$
	and $|r| \in [3/\delta, M]$ for all $r \in \cR$.
	Suppose further that for some $M^* \geq 12$, 
	every reduced word $u \in \langle \cS \rangle$ of length $M^*$
	appears at least once in some cyclic conjugate of some relator $r^{\pm 1}, r \in \cR$.
	Then for some universal constant $C>0$, we have
	\[
		\Cdim(\bdry G) \geq 1 + C\log(2m) \cdot \frac{M^*}{\log(M/M^*)}.
	\]
	(If we have a $C'(\frac{1}{11})$ presentation, the lower bound on
	the lengths of relators holds automatically.)
\end{theorem}
In the induction step \cite[Section 5.2.2]{Mac-12-random-cdim}, 
we split the peripheral path $E$ into segments of lengths between $3$ and $6$, and
extended geodesics of length $\lfloor M^*/2 -3 \rfloor$ from the endpoints
of these segments.  This resulted in a $\log(M)$ in the denominator of the lower bound
because each face had $\leq M$ new faces attached to it.

To prove Theorem~\ref{thm-lower-bound-sc} for large $M^*$,
split $E$ into segments of lengths between
$M^*/6$ and $M^*/3$, and extend geodesics a length $\lfloor M^*/3-3 \rfloor$.
Since each face has $\leq 6M/M^*$ new faces attached to it,
the denominator of the lower bound is replaced by $M/M^*$.

This result gives a sharp bound on conformal dimension for
random groups at densities $d<\frac{1}{16}$, but at higher
densities we need our new tools.
Note that Theorem~\ref{thm-lower-bound-sc} does not give anything new in the few relator
model, as in that situation $M^*$ is of the order of $\log l$.

\subsection{Diagrams in the density model}
We use the following isoperimetric inequality of Ollivier to rule out the existence
of certain diagrams and show that geodesic bigons have a specific form
(cf.\ \cite[Section 3]{Mac-12-random-cdim}).

\begin{theorem}[{\cite[Theorem 1.6]{Oll-07-sc-rand-group}}]\label{thm-O-sc-rand}
	At density $d$, for any $\eps >0$ the following property occurs a.a.s.:
	all reduced van Kampen diagrams $D$ satisfy
	\begin{equation*}
		|\partial D| \geq (1-2d -\eps) l|D|.
	\end{equation*}
\end{theorem}
This theorem is used extensively by Ollivier and Wise to control the geometry of van Kampen diagrams
in random groups at low densities.  We collect some of these properties now.
\begin{proposition}[Ollivier--Wise]
	\label{prop-OW-12cells}
	In the Cayley complex of a random group at density $d<\frac14$, a.a.s.\ we have that
	the boundary path of every 2-cell embeds, and
	the boundary paths of any two distinct 2-cells meet in a connected (or empty) set.
\end{proposition}
\begin{proof}
	Follows from Proposition 1.10 and Corollary 1.11 of~\cite{Oll-Wis-11-rand-grp-T}.
\end{proof}
\begin{lemma}\label{lem:rand-geodesics-meet-cell}
	If $G$ is a random group at density $d<\frac14$, then a.a.s., if
	$\gam$ is a geodesic in the Cayley graph $\Gam=X^{(1)}$ of $G$, for every
	$2$-cell $R \subset X$, $R \cap \gam$ is connected.
\end{lemma}
\begin{proof}
	If not, we can find a loop in $\Gam$ consisting of a connected path $\beta \subset \partial R$
	of length $\leq l/2$, and a subgeodesic $\gam' \subset \gam$ of length $|\gam'| \leq l/2$
	that meets $R$ only at its endpoints.
	So $\beta \cup \gam'$ has length at most $l$, which is only
	possible if $\beta \cup \gam'$ bounds a single $2$-cell $R'$
	by \cite[Proposition 1.10]{Oll-Wis-11-rand-grp-T}.
	Since $|R \cap R'| = |\beta| = l/2 \geq 2dl$, by Theorem~\ref{thm-O-sc-rand}
	we have $R=R'$ and $\gam' \subset \partial R$, a contradiction.
\end{proof}

In a reduced diagram $D$, a 2-cell $R$ is called a \emph{pseudoshell} if 
$|\partial R \cap \partial D| > \frac{1}{2}|\partial R|$.
\begin{theorem}[{\cite[Theorem 5.1]{Oll-Wis-11-rand-grp-T}}]\label{thm-OW-three-pshells}
	For a random group at density $d<\frac{1}{6}$, a.a.s.\ every reduced diagram with at least
	three $2$-cells has at least three pseudoshells.
\end{theorem}

We use this theorem to show that any reduced diagram for a geodesic bigon at
densities $d<\frac{1}{6}$ has the following special form, which is a variation
on the ladders considered in Sections~\ref{sec-wise-walls}--\ref{sec-new-upper}.
\begin{definition}
	A connected disk diagram $D$ is a \emph{ladder 
	(from $\beta\subset \partial D$ to $\beta' \subset \partial D$)},
	if $D$ is a union of a sequence $R_1, R_2, \ldots, R_k$,
	for some $k \in \N$,
	where each $R_i$ is a closed $1$-cell or a closed $2$-cell,
	and $R_i \cap R_j = \emptyset$ for $|i-j|>1$.
	Moreover, $\beta$ and $\beta'$ are closed paths in
	$R_1 \setminus R_2$ and $R_k \setminus R_{k-1}$, respectively.
\end{definition}
See Figure~\ref{fig-bigon} for a ladder where $\beta$ and $\beta'$ are points.
\begin{figure}
	\begin{center}
	\includegraphics[width=0.9\textwidth]{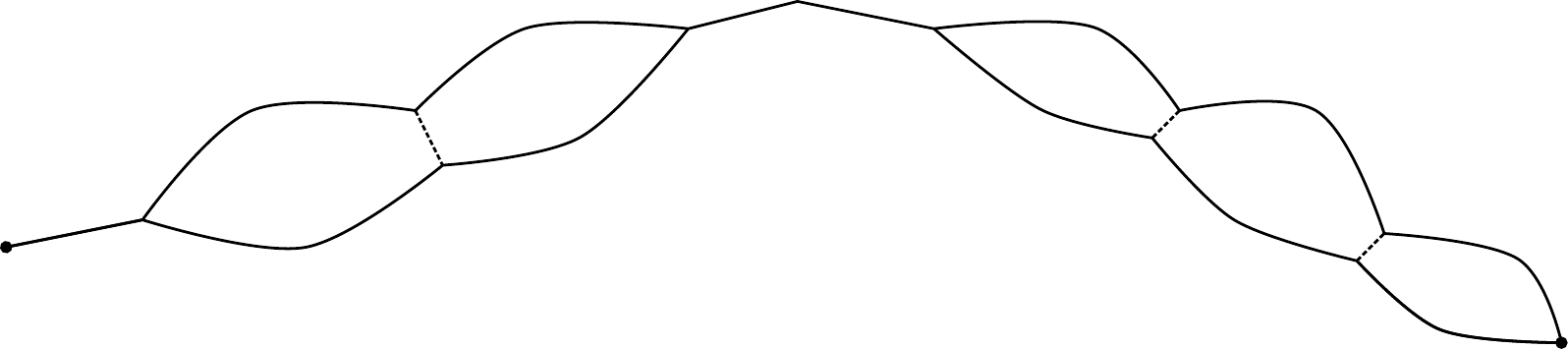}
	\end{center}
	\caption{A ladder}\label{fig-bigon}
\end{figure}

\begin{lemma}[{cf. \cite[Lemmas 3.11, 3.12]{Mac-12-random-cdim}}]\label{lem-bigons}
	Let $G$ be a random group at density $d<\frac{1}{6}$.  Then a.a.s.\ we have the following
	properties, for any reduced diagram $D \ra X$.
	
	(i) Suppose $\partial D$ consists of, 
	in order, a geodesic $\gam_1$,	a path $\beta \subset R \subset D$,
	a geodesic $\gam_2$, and a path $\beta' \subset R' \subset D$.
	Suppose further that $\partial R$ and $\partial R'$ each contain edges from both $\gam_1$ and $\gam_2$.
	Then $D$ is a ladder $R=R_1, R_2, \ldots, R_r=R'$ from $\beta$ to $\beta'$.
	
	(ii) Suppose $\partial D$ consists of a geodesic $\gam_1$,
	a path $\beta \subset \partial D$, and a geodesic $\gam_2$,
	with $\gam_1, \gam_2$ sharing the endpoint $p$.
	Suppose further that $\beta = D \cap R$ for some $2$-cell $R$ in $X$,
	i.e.\ that $D \cup R \ra X$ is also a reduced diagram.
	Then $D$ is a ladder $R_1, R_2, \ldots, R_k$ from $p$ to $\beta$.
	
	(iii) Suppose $\partial D$ is a geodesic
	bigon $\gam_1 \cup \gam_2$ with endpoints $p$ and $q$.
	Then $D$ is a ladder from $p$ to $q$. 
\end{lemma}
At densities $d>\frac{1}{6}$, geodesic bigons need not be ladders; we see
situations such as Figure~\ref{fig-ladderfail}.
\begin{figure}
		\centering
		\def\svgwidth{0.6\columnwidth}
		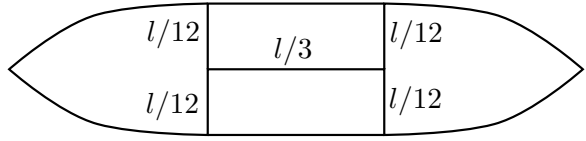
		\caption{Example of geodesic bigon at density $d>\frac{1}{6}$}\label{fig-ladderfail}
	\end{figure}
\begin{proof}
	In case (ii), if $\beta$ is not trivial, then we may assume that $D$ contains $R$.
	This is because if it does not, we can glue $R$ to $D$ along $\beta$ 
	to give a new reduced diagram $D'$.
	The boundary of $D'$ consists of $\gam_1,\overline{\partial R\setminus \beta}, \gam_2$,
	and if $D'$ is a ladder from the shared endpoint of $\gam_1, \gam_2$ to $\overline{\partial R\setminus \beta}$,
	then $D$ is a ladder from the shared endpoint of $\gam_1, \gam_2$ to $\beta$.
	
	We prove the lemma (in all three cases simultaneously)
	by induction on the number of cells in $D$, where
	$\partial D$ has the required form, and if $\beta$ (respectively $\beta'$) is non-trivial, then
	$D$ contains $R$ (respectively $R'$).
	
	If $\gam_1$ and $\gam_2$ meet at a vertex $x \in D$ other than their endpoints, 
	then we can split the diagram $D$ at $x$ into two diagrams $D_1$ and $D_2$ of fewer
	cells than $D$.
	By induction, both $D_1$ and $D_2$ are ladders, 
	and so $D$ is also a ladder.
	
	If $D$ has at most two $2$-cells, we are done by Proposition~\ref{prop-OW-12cells}.
	
	If $D$ has at least three $2$-cells, then by Theorem~\ref{thm-OW-three-pshells},
	$D$ has at least three pseudoshells.
	Now, as boundaries of $2$-cells embed (Proposition~\ref{prop-OW-12cells}), and the geodesics
	are disjoint except at their endpoints, we have that $D$ is a topological disk.
	Therefore, apart from $R$ and $R'$, there must be at least one other pseudoshell $R''$
	in $D$, and $R''$ must meet both $\gam_1$ and $\gam_2$ (see Figure~\ref{fig-bigon2}).
	By induction, the subdiagrams of $D$ bounded by $R$ and $R''$, and by $R''$ and $R'$,
	are both ladders, so $D$ is a ladder as well.
\end{proof}
\begin{figure}
	\centering
	\def\svgwidth{0.9\columnwidth}
	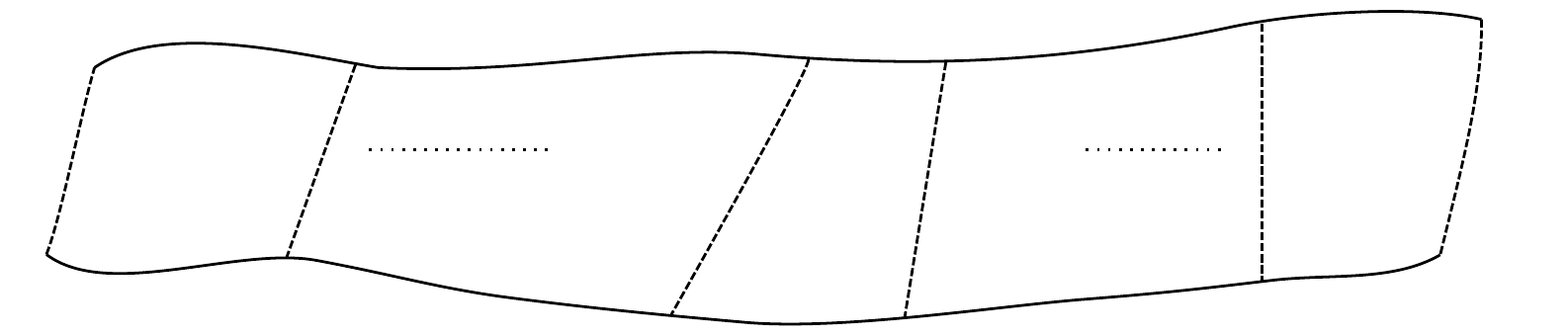
	\caption{}\label{fig-bigon2}
\end{figure}

\subsection{Extending geodesics}

We can extend geodesics in random groups at density $d<\frac{1}{6}$ in multiple ways,
and ensure that the extensions do not start with a long word from a relator
(cf.\ \cite[Section 5.1]{Mac-12-random-cdim}).

First we show that near any point in the Cayley graph of a
random group at density $d<\frac{1}{6}$ 
there is at most one point with at least two non-extending neighbours
(Definition~\ref{def:nonextendingneighbours}).
\begin{lemma}[cf.\ Lemma~\ref{lem:geod-extend-sc}]\label{lem-geod-extend-rand}
	Suppose $G$ is a random group at density $d<\frac{1}{6}$ with Cayley graph $\Gam$,
	and $\eta \in \Z_{\geq 0}$ is given.
	Given $a \in G$, let $E_{a, \eta}$ be the collection of points
	$b \in G$ so that $d(1,b) = d(1,a)+d(a,b)$, $d(a,b) \leq \eta$,
	and $b$ has at least two non-extending neighbours.
	
	Then a.a.s.\ for every point $a \in G$, 
	either $E_{a, \eta} = \emptyset$, or $E_{a, \eta} = \{b\}$ where $b$ has
	exactly two non-extending neighbours.
	
	In fact, if we fix $\eps < \frac{1}{6}-d$, 
	the same conclusion holds a.a.s.\
	for all $a \in G$ and all $\eta < \min\{ \frac{1}{6}l-1, \frac{3}{2}\eps l-1\}$.
\end{lemma}
\begin{proof}
	Suppose we have a group presentation with $l > 6\eta+6$
	and a point $a \neq 1$ in the corresponding Cayley graph which contradicts our desired conclusion.
	That is, either we have two distinct points $b_1, b_2 \in E_{a, \eta}$, or 
	$E_{a, \eta} = \{b\}$ and $b$ has three non-extending neighbours.
	In this last case, set $b_1=b_2=b$.  
	The strategy of proof is to use these points to build a reduced diagram which contradicts
	Theorem~\ref{thm-O-sc-rand}.
	
	Let $\gam_0 = [a, 1]$ be a geodesic.  
	For $i=1,2$, let $\gam_i$ be a geodesic from $b_i$ to $1$ extending $\gam_0$, 
	and denote by $b_i' \in \gam_i$ the point for which $d(b_i,b_i')=1$.
	(Observe that if $b_1=b_2$, then $b_1'=b_2'$.)
	
	By our assumption, there are two points $c_1 \neq c_2$, so that $c_i \neq b_i'$,
	$d(c_i,b_i)=1$, and $d(1,c_i) \leq d(1,b_i)$, for $i=1,2$.
	(In the case that $b_1 \neq b_2$, we must have $c_1 \neq c_2$, else we have a relation
	in $G$ of length $\leq 2\eta+2 < l$.  In the other case, $b$ has at least three 
	non-extending neighbours.)
	
	For each $i=1,2$, let $\beta_i$ be a geodesic from $c_i$ to $1$, which necessarily does
	not pass through $b_i'$.
	Let $D_i$ be a reduced diagram for the geodesic triangle $\gam_i, [b_i,c_i], \beta_i$,
	for $i=1,2$.
	By the choice of paths, for each $i=1,2$, 
	the edge $[b_i,c_i]$ is disjoint from $\gam_i, \beta_i$ in $X$, so $[b_i,c_i]$
	must lie in the boundary of some $2$-cell $R_i \subset D_i$.
	So the diagram $D_i \ra X$ is a ladder by Lemma~\ref{lem-bigons}(i) 
	applied with $\partial D_i$ consisting of, in order, $\gam_i, [b_i,c_i], \beta_i$ 
	and a trivial path.
	
	As follows from Theorem~\ref{thm-O-sc-rand}, or Lemma~\ref{lem-multiplerelators} below,
	any two 2-cells in a reduced diagram meet along a segment 
	of length $\leq l/3$.
	
	In the case of the diagrams $D_i$, this implies that $a \in \partial R_i$.
	This is because $\beta_i$ is a geodesic, so
	$|\partial R_i \cap \gam_i| \geq \frac12 l - \frac13 l -1 \geq \eta \geq d(b_i,a)$.
	
	Now consider the diagram $D = D_1 \cup_{\gam_1 \cap \gam_2} D_2 \ra X$
	formed by gluing the diagrams
	$D_1$ and $D_2$ along $\gam_1 \cap \gam_2 \supseteq \gam_0$.
	Let $p$ be the first point of the path $\gam_0$ to meet $\gam_1$ or $\gam_2$ after $a$,
	and discard from $D$ any cells which are not in the boundary of a 2-cell that meets
	$\gam_0[a,p]$.
	
	Note that $c_1 \neq c_2$ implies that $R_1$ and $R_2$ are different 2-cells in $X$,
	and so in $D$ they do not reduce.
	Any reductions in $D$ that do occur must happen along $\gamma_0$.
	Perform the reduction which occurs closest to $a$ to find a reduced diagram $D'$ as in
	Figure~\ref{fig-geodextend}(a), up to swapping the indices $1$ and $2$.
	\begin{figure}
		\centering
		\def\svgwidth{0.95\columnwidth}
		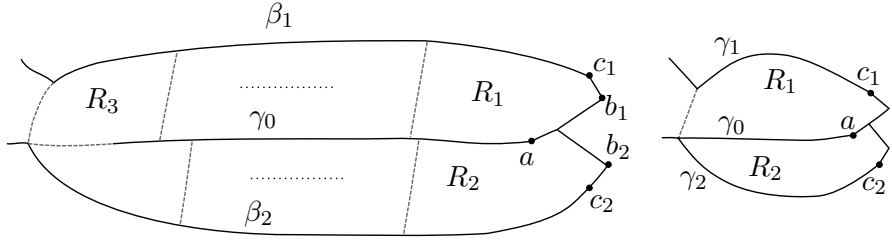
		\caption{Extending geodesics (a) and (b)}\label{fig-geodextend}
	\end{figure}
	We choose $\eps>0$ so that $2d+2\eps < \frac13$.  Then by Theorem~\ref{thm-O-sc-rand},
	we have, a.a.s., that every reduced diagram $D$ satisfies $|\partial D| > (\frac23+\eps)l|D|$.
	
	Suppose $|D'| = 3+k$, for some $k \geq 0$.  Then Theorem~\ref{thm-O-sc-rand} gives
	$|\partial D'| > (\frac23+\eps)l(3+k) = 2l+\frac23 kl+\eps l(3+k)$.
	
	On the other hand, $|\partial D'| \leq 2(\frac12 l+1+\eta)+l+\frac12 lk = 2+2\eta+2l+\frac12 kl$, 
	as $R_1,R_2$ each contribute at most $\frac12 l+1+\eta$ to $|\partial D'|$, 
	$R_3$ contributes at most $l$, and all other faces contribute at most $l/2$.
	So
	\[
		2l+\tfrac12 kl+3\eps l \leq 2l+\tfrac23 kl+\eps l(3+k) < 2+2\eta+2l+\tfrac12 kl,
	\]
	which is a contradiction for $l > (2+2\eta)/(3\eps)$.
	
	If $|D'|=2$, we are in the situation of Figure~\ref{fig-geodextend}(b).
	As $\gam_2$ is a geodesic, $|\partial R_1 \cap \partial R_2| \geq \frac12 l-1-\eta$,
	but then $\frac23 l\cdot 2 < |\partial D'| \leq 2(\frac12 l+1+\eta)$,
	a contradiction for $l \geq 6+6\eta$.
	
	Thus a.a.s.\ no such point $a$ exists in $\Gamma$.
\end{proof}
In a similar, but slightly simpler fashion, we see the following.
\begin{lemma}\label{lem-twoinitalsegs}
	In a random group at density $d<\frac{1}{6}$, a.a.s.\ for any $a \in G$, there are
	at most two initial segments of length $l/6$ for a geodesic $[a,1]$.
\end{lemma}
\begin{proof}
	Suppose otherwise, that $\gam_0,\gam_1,\gam_2$ are three geodesics from $a$ to $1$,
	and have different initial segments of length $l/6$.
	We may assume that $\gam_2$ branches away from $\gam_0$ no earlier than $\gam_1$ does.
	
	As in the previous lemma, $\gam_0$ and $\gam_i$ form a reduced diagram $D_i$ which is a ladder.
	We glue these diagrams along $\gam_0$, perform the reduction closest to $a$,
	and discard the rest of the diagram after this point, or after the point at which 
	one of $\gam_1$ or $\gam_2$ has been reunited with $\gam_0$.  Call this diagram $D$ 
	(see Figure~\ref{fig-initial-segments}(a)).
	\begin{figure}
		\centering
		\def\svgwidth{0.95\columnwidth}
		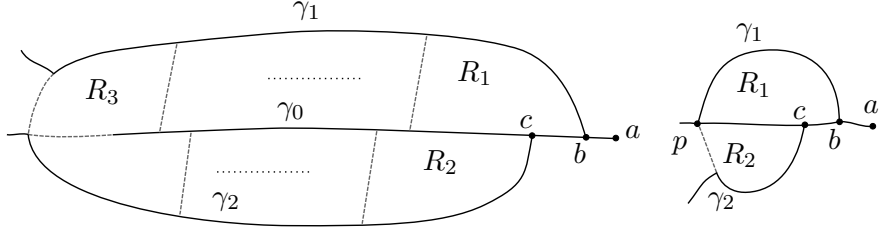
		\caption{Initial segments (a) and (b)}\label{fig-initial-segments}
	\end{figure}	
	
	Now, by Theorem~\ref{thm-OW-three-pshells}, if $|D| \geq 3$, it has at least three pseudoshells.
	But the only possible pseudoshells are $R_1$ and $R_3$, so we must have $|D| = 2$,
	which is illustrated by Figure~\ref{fig-initial-segments}.
	
	Theorem~\ref{thm-O-sc-rand} gives us $C'(\frac{1}{3})$,
	so $[c,p]\subset \gam_0$ has length less than $l/3$.	
	As $\gam_1$ is a geodesic, the path from $b$ to $c$ in $\gam_0$
	has length $\geq l/2-l/3 = l/6$.
	This gives the desired result: a second branching of geodesics from $a$ to $1$ can
	occur only after distance $l/6$.
\end{proof}

\begin{proposition}[{cf.\ \cite[Lemma 5.4]{Mac-12-random-cdim}}]\label{prop-startrel}
	Given $d<\lambda<\frac{1}{6}$, there exists $\eta = \eta(d,\lam) \in \N$ so that, a.a.s.,
	for every $u' \in \Gam$ there exists $u \in \Gam$ with $d(u',u)=\eta$,
	$d(1,u)=d(1,u')+d(u',u)$, and so that no geodesic $[u,1]$ starts with a subword of 
	some relator of length greater than $\lambda l +\eta$.
	We can further require that the initial edge of $[u',u]$ is not a specified
	edge adjacent to $u'$.
\end{proposition}
We will only use this Proposition with $\lambda = \frac{1}{8}$.
\begin{proof}
	By Lemma~\ref{lem-geod-extend-rand}, there is at most one point $v \in E_{u',\eta}$
	such that $d(1,v) = d(1,u')+d(u',v)$ and $0< d(u',v) \leq \eta$,
	and so that the first edge of $[v,1]$ is not unique.
	
	If $v=u'$ (or $E_{u',\eta} = \emptyset)$), we can extend geodesics at $u'$
	in at least $(2m-2)\geq 2$ ways, at least one of which avoids a specified edge.
	On the other hand, if $v \neq u'$, then 
	we can extend geodesics at $u'$ in $(2m-1)\geq 3$ ways, 
	at least one of which avoids both a specified edge and the direction leading to $v$.
	
	In either case, we can then further extend geodesics to 
	$(2m-1)^{\eta-1} \geq 2^{\eta-1}$
	different points $u$ with $d(u',u)=\eta$, and
	$d(1,u) = d(1,u')+d(u',u)$, so that every geodesic $[u,1]$ goes through $u'$.
	(These points really are distinct because $E_{u',\eta}$ consists of
	at most the point $v$.)
	
	Suppose from each of these points $u$ we have a geodesic $\gam_u$ from $u$ to $1$,
	which starts with an initial segment of a relator $r_u$ of length $\eta+\lam l$.
	As the points are distinct, these relators are distinct (that is, the corresponding $2$-cells
	are distinct in $X$).
	
	There are at most two different initial segments of length $l/6$ for geodesics 
	from $u'$ to $1$ by Lemma~\ref{lem-twoinitalsegs}.
	Therefore at least $2^{\eta-2}$ of the relators $r_u$ share an identical geodesic segment 
	from $u'$ of length $\lam l$.
	This is a contradiction to Lemma~\ref{lem-multiplerelators} below, provided $\eta > \eta(d,\lam)$.
\end{proof}

The following lemma, which we used above,
gives a sufficient condition to prevent some
word of length $> ld$ appearing in multiple ways in the relators.
The case $N=2$ reproves the fact that if $2d < \lambda$, then
a random group at density $d$ is $C'(\lam)$.
\begin{lemma}\label{lem-multiplerelators}
	Given $0 < d < \lambda < \frac{1}{2}$, for $N > \lam/(\lam-d)$, a.a.s.\ for a random group $G$ at
	density $d$ there is no word of length $\lambda l$ which appears in $N$ different ways
	in the relators of $G$.
\end{lemma}
\begin{proof}
	Suppose a word $w$ of length $\lambda l$ appears in $N$ different places in the relators
	$\cR$ of a random group with $m$ generators, at length $l$ and density $d$.
	
	Write $N = k_1+k_2+\cdots +k_t$, where $k_1 \geq k_2 \geq \cdots \geq k_t \geq 1$,
	and consider the situation that $w$ appears in $t$ distinct relators
	$r_1, \ldots r_t$, and in each $r_i$ it appears in $k_i$ different ways.
	
	Let $P = P(k_1, \ldots, k_t)$ be the probability of this occurring.
	We bound $P$ by considering three different cases.
	
	In the first case, $k_1=1$ and so every $k_i=1$, implying that $t=N$.
	There are at most $(2m-1)^{ldN}$ choices for the relators $r_1, \ldots, r_N$,
	and $(2l)^N$ possible starting points for the designated subwords of length $\lambda l$
	in each $r_i^{\pm 1}$.
	There is no restriction on the relator $r_1$, but for each $i\geq 2$, the
	probability that the subword from $r_i$ matches the given subword from $r_1$ is at most
	$2 (2m-1)^{-\lambda l}$.  (The factor of $2$ deals with minor issues due to cyclically
	reduced words; see \cite[Lemma 2.4]{Mac-12-random-cdim} for an even sharper bound.)
	
	Thus in this case, we have
	\[
		P \leq (2m-1)^{ldN} (2l)^N \left( 2 (2m-1)^{-\lambda l} \right)^{N-1},
	\]
	which goes to zero as $l \ra \infty$, using our assumption that $N > \lambda/(\lambda-d)$.
	
	In the second case, we have $k_1 = 2$,
	which gives that $t \geq \lceil N/2 \rceil$.
	Let $P_1$ be the probability that some relator $r_1$ has a subword of 
	length $\lam l$ appearing in two different ways.
	There are $(2m-1)^{ld}$ choices for $r_1$, and $(2l)^2$ choices for the positions
	of the two different subwords $u_1$ and $u_2$ in $r_1^{\pm 1}$.
	
	If $u_1$ and $u_2$ do not overlap, then the choice of $u_2$ is entirely determined
	by the choice of $u_1$, and the two words agree with probability at most
	$2^2 (2m-1)^{-\lam l}$  (the factor of $2^2$ arises from filling in two cyclically
	reduced words).
	
	If $u_1$ and $u_2$ overlap, but have the same orientation, for example $u_2$ starts
	from the $j$th letter of $u_1$, then $u_2$ is entirely determined by the initial $j$
	letters of $u_1$, so the probability that $r_1$ has the required form is at most
	$2 (2m-1)^{-\lam l}$.
	
	If $u_1$ and $u_2$ overlap, but in the opposite orientation, then as
	we choose letters for $u_1$, there are at least $\lam l/2$ free choices which
	we make that determine the corresponding letter of $u_2$ (the worst case situation
	is when $u_1$ and $u_2$ almost entirely overlap).
	So, after again adding a factor of $4$ to deal with cyclically reduced words,
	we have that this occurs with probability at most 
	$4(2m-1)^{- \lam l/2}$.
	
	Putting this together, we have $P_1 \leq (2m-1)^{ld}(2l)^2 4(2m-1)^{-\lam l/2}$.
	
	The probability that $r_i$, for $i \geq 2$ contains $k_2 \geq 1$ different 
	copies of the given subword is at most the probability that it contains one such
	subword, which is at most $(2m-1)^{ld}(2l)2(2m-1)^{-\lam l}$.
	Thus,
	\begin{align*}
		P & \leq (2m-1)^{ld}(2l)^2 4(2m-1)^{-\lam l/2} \left( 
			(2m-1)^{ld}(2l)(2m-1)^{-\lam l} \right)^{t-1} \\
		 & = 4 (2l)^{t+1} (2m-1)^{-((\lam-d)t-\lam/2)l}.
	\end{align*}
	Because $(\lam-d)t \geq (\lam -d)N/2 > \lam/2$,
	this bound goes to zero as $l \ra \infty$.
	
	Finally, in the third case $k_1 \geq 3$.
	Here, the probability that that some $r_1$ contains $k_1$ subwords
	of length $\lam l$ is bounded by the probability that some $r_1$ contains two matching subwords
	of length $\lam l$ which are either disjoint or have the same orientation.
	As we saw above, this is bounded by $(2m-1)^{ld}(2l)^2 4 (2m-1)^{-\lam l}$.
	
	As in the second case, the probability that $r_i$, for $i \geq 2$ contains $k_2 \geq 1$ different 
	copies of the given subword is at most $(2m-1)^{ld}(2l)2(2m-1)^{-\lam l}$.
	Thus we have
	\begin{align*}
		P & \leq (2m-1)^{ld}(2l)^2 4 (2m-1)^{-\lam l} \left(
			(2m-1)^{ld}(2l)2(2m-1)^{-\lam l} \right)^{t-1} \\
		  & = (2l)^{t+1} (2m-1)^{-(\lam-d)tl} \ra 0 \ \text{ as } \ l \ra \infty.
	\end{align*}
	
	As the number of ways to write $N$ as $N=k_1+ \cdots +k_t$ is bounded independently of $l$,
	the proof is complete.	
\end{proof}

\subsection{Building a round tree}

We follow the outline of \cite[Section 5.2]{Mac-12-random-cdim}),
and build a round tree $A = \bigcup_{\ba \in T^\N} A_{\ba}$
(see Definition~\ref{def:comb-round-tree}).
Our vertical branching is controlled by the index set $T=\{1,2,\ldots,(2m-1)^{K-\eta-1}\}$,
where $K = \lfloor M^*/3 \rfloor$.
Here $M^*$ is chosen so that $\cR$ contains every word of length $M^*$ as a subword of some 
relator.

We build the round tree by induction.
At step $0$, let $A_0 = \{A_\emptyset\}$,
where $A_\emptyset$ is a $2$-cell which contains the base point $1$ in its boundary.
Let $L_\emptyset$ and $R_\emptyset$ be the two edges of $\partial A_\emptyset$ which meet $1$,
and let $E_\emptyset$ be the rest of $\partial A_\emptyset$.

At step $n+1$, we assume we have $A_n = \bigcup_{\ba_n \in T^n} A_{\ba_n}$, with each
$A_{\ba_n}$ a planar $2$-complex with boundary consisting of a left path $L_{\ba_n}$ from $1$,
a right path $R_{\ba_n}$ from $1$ and an outer boundary $E_{\ba_n}$.
We split $E_{\ba_n}$ into subpaths of lengths 
between $K/2$ and $K$,
so that local minima of $d(1,\cdot)|_{E_{\ba_n}}$ are not endpoints of the subpaths.

From each endpoint $u'$ of the subpaths, as it is not a local minima 
of $d(1,\cdot)|_{E_{\ba_n}}$, there is at most one adjacent edge in $E_{\ba_n}$
which extends geodesics to $1$.  Ruling out that direction, we use
Proposition~\ref{prop-startrel} (with $\lam = \frac{1}{8}$)
to extend a geodesic $[1,u']$ to the point $u$ with $d(u',u) = \eta$,
and so that the extension $[u',u]$ meets $E_{\ba_n}$ only at $u'$.

Since $d < \frac{1}{8}$, for $\eps = \frac{1}{24}$ we have
$\eps < \frac{1}{6}-d$.  We will later set $M^* = \lceil \frac{4}{5} dl \rceil$,
so $K-\eta \leq \frac{1}{3} M^* < \frac{1}{3} \cdot \frac{4}{5} \cdot \frac{1}{8} l = \frac{1}{30}l$,
which is certainly less than $\min\{\frac{1}{6}l-1, \frac{3}{2}\eps l-1\} = \frac{1}{16}l-1$.
So Lemma~\ref{lem-geod-extend-rand} (applied with $a=u$ and ``$\eta$'' equal to $K-\eta$)
says that there are at least
$2m-2 \geq 2$ ways to extend from $u$ so that all $|T|=(2m-1)^{K-\eta-1}$
possible further extensions of length $K-\eta-1$ will give geodesics
going away from $1$.
Choose one of these initial directions, and so find
$|T|$ points at distance $K$ from $u'$, and distance $d(1,u')+K$ 
from $1$.

For each $a_{n+1} \in T$, we have a corresponding geodesic segment of 
length $K$ leaving each endpoint
of a subpath of $E_{\ba_n}$.  
Since $3K \leq M^*$, we can fill in a $2$-cell along adjacent
geodesic segments and the path between them.
This defines $A_{\ba_{n+1}}$, with $\ba_{n+1} = (\ba_n, a_{n+1})$.

This inductive construction defines an infinite polygonal 
complex $A$, along with a natural immersion 
$i:A \ra X$.

\begin{remark}
	The perfect matching approach of Section~\ref{sec-new-lower-sc} gives some improvement
	in $C$ in Theorem~\ref{thm-main-density}, but as we have not achieved $C \ra 1$,
	we content ourself with the simpler approach above.
\end{remark}

\begin{lemma}[{cf.\ \cite[Lemma 5.6]{Mac-12-random-cdim}}]\label{lem-top-embed}
	The map $i:A \ra X$ is a topological embedding.
	
	More precisely, for every $p \in A$, every geodesic joining $i(p)$ to $i(1)=1$ in $\Gam=X^{(1)}$
	is the image under $i$ of a (geodesic) path joining $p$ to $1$ in $A^{(1)}$.
\end{lemma}
\begin{proof}
	We prove the lemma by induction on $n$, where $p$ is a point in the boundary of some 2-cell $R$ 
	added at stage $n$ of the construction.
	
	If $n=0$, then $p\in A_0$.  If some geodesic $[p,1]$ did not lie in $\partial A_0$,
	then there would be a contradiction to Lemma~\ref{lem:rand-geodesics-meet-cell}.
	
	For the inductive step, suppose $u', v' \in E_{\ba_n}$ are consecutive endpoints of segments 
	in the construction of $A_{n+1}$, for some $\ba_{n+1} \in T^{n+1}$.
	Let $u,v$ denote the corresponding points in $E_{\ba_{n+1}}$, and denote by $\gam_{uv}$
	the path connecting them, which lies in $\partial R$ for some $R \subset A_{n+1}$.  
	
	Suppose $p \in \partial R\setminus A_n$.
	It suffices to show that every geodesic $[p,1]$ comes from a subpath of $\partial R$ followed
	by a geodesic $[u',1]$ or $[v',1]$, which, by induction, is the image of a geodesic in $A_n$.
	
	Suppose this does not hold, and there is some geodesic $\beta$ from $p$ to $1$ not of this form.
	Without loss of generality, we assume that the first edge of $\beta$ is not in $\partial R$.
	Suppose $d(v',p) \leq d(u',p)$, and let $\alpha$ be the (geodesic) subpath from $p$ to $v'$ 
	in $\partial R$.
	Let $D'$ be a reduced diagram with boundary $[v',1], \beta, \alpha$.
	By Lemma~\ref{lem-bigons}(ii), $D'$ has the form of a ladder, and as $\beta$ branches
	off from $R$ at $p$, $p$ is contained in the boundary of some 2-cell $R' \subset D'$.
	
	Now let $D$ be the diagram formed from $D'$ by attaching $R$ along $\alpha$ 
	(Figure~\ref{fig-topembed1}).
	Again, as $\beta$ branches away from $R$ at $p$, we have that $D$ is reduced.
	So again by Lemma~\ref{lem-bigons}, we have that $D$ is a ladder.
	Let $k_2 = |\alpha|$ be the length of the intersection of $R$ and $R'$.
	Let $k_3$ be the length of $\partial R' \cap \partial D$ which is not in $\beta$.
\begin{figure}
	\centering
	\def\svgwidth{0.9\columnwidth}
	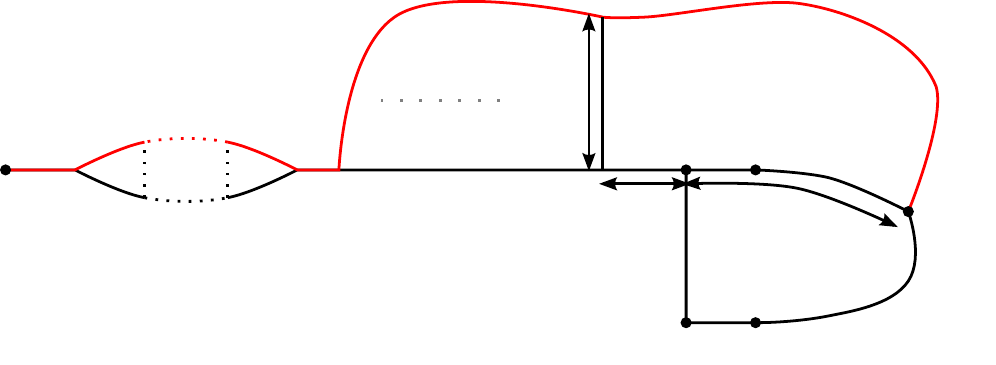
	\caption{Controlling geodesics to the identity}\label{fig-topembed1}
\end{figure}
	
	If $R'$ has no adjacent 2-cell in $D'$, then as $\beta$ is a geodesic we have
	$k_2+k_3 \geq l/2$.
	Theorem~\ref{thm-O-sc-rand} gives us $C'(\frac{1}{4})$, 
	so $k_2 \leq l/4$, and by the construction and
	Proposition~\ref{prop-startrel}, $k_3 \leq l/8$, a contradiction.
	
	If $R'$ has an adjacent 2-cell $R''$ in $D'$, let $0 \leq k_1 \leq l/2$ 
	denote the length of the intersection of $R'$ with $R''$.
	As $\beta$ is a geodesic, again we have $k_1+k_2+k_3 \geq l/2$.
	But Theorem~\ref{thm-O-sc-rand} applied to the diagram $D''$ formed by
	$R \cup R' \cup R''$ gives that
	\begin{align*}
		\tfrac94 l & < 3(1-2d-\eps)l \leq |\partial D''| = 3l-2k_1-2k_2 \\
		& = 3l-2(k_1+k_2+k_3)+2k_3 \leq 3l-l+\tfrac14 l= \tfrac94 l,
	\end{align*}
	a contradiction.
\end{proof}
\begin{remark}
	This last lemma is the key point at which $d < \frac{1}{8}$ was used,
	rather than just $d<\frac{1}{6}$.
\end{remark}

\begin{lemma}[{cf.\ \cite[Lemma 5.7]{Mac-12-random-cdim}}]\label{lem-qi-embed}
	Denote the path metric on $A^{(1)}$ 
	by $d_A$. 
	Then $i: (A^{(1)}, d_A) \ra (X, d)$ is a quasi-isometric embedding.
\end{lemma}
\begin{proof}
	We denote by $d_X$ the pullback metric on $A^{(1)}$, where $d_X(x,y) = d(i(x),i(y))$.
	
	Now take any $x,y \in A^{(1)}$.  
	Since $i$ sends edges in $A$ to edges in $X$, 
	we have $d_A(x,y) \geq d_X(x,y)$.
	
	Consider the geodesic triangle with vertices $1, x, y$ and edges $\gam_{1x}, \gam_{1y}, \gam_{xy}$.
	Let $D$ be a reduced diagram for this triangle, and as usual remove 
	all vertices of degree $2$.
	Let $R_{xy}$ be the cell which meets both $\gam_{1x}$ 
	and $\gam_{1y}$ and is furthest from $1$ in $D$
	(usually this will be a 2-cell).  
	Lemma~\ref{lem-bigons} shows that $R_{xy}$ bounds a ladder in $D$ containing $1$.
	In a similar way, define $R_{1x}$ and $R_{1y}$ as the last cells meeting the corresponding two geodesics,
	see Figure~\ref{fig-qiembed1}.
\begin{figure}
	\centering
	\def\svgwidth{0.7\columnwidth}
	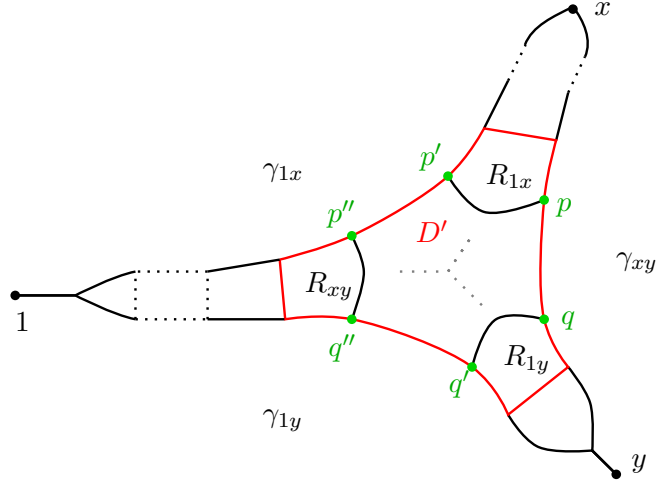
	\caption{Quasi-isometric embedding}\label{fig-qiembed1}
\end{figure}
	
	Let $D' \subset D$ be the subdiagram bounded by, and containing, $R_{xy}$, $R_{1x}$ and $R_{1y}$.
	Let $t$ be the number of 2-cells in $D'$.  Apart from at most three 2-cells, every 2-cell in $D'$ meets
	$\partial D'$ along at most half its boundary.
	Therefore by Theorem~\ref{thm-O-sc-rand},
	\[
		\tfrac34 lt < (1-2d-\eps)lt \leq |\partial D'|
			\leq 3l+(t-3)\cdot \tfrac12 l = \tfrac12 lt+\tfrac32 l,
	\]
	so $t < 6$, and thus $t \leq 5$.  
	
	Label $p, p', p''$ and $q,q', q''$ as in Figure~\ref{fig-qiembed1}.
	The triangle inequality and this bound on $t$ show that:
	\begin{equation}
	\begin{split}
		d_X(x,y) & = d_X(x,p)+d_X(p,q)+d_X(q,y)  \\
			& \geq d_X(x,p'')-d_X(p,p'') + d_X(y,q'')-d_X(q,q'') \\
			& \geq d_X(x,p'')+ d_X(y,q'')-5l. 
	\end{split}\label{eq-dGam-bound}
	\end{equation}
	
	We now bound $d_A(p'',q'')$.
	Recall that $R_{xy}$ bounds a ladder in $D$ containing $1$.
	We suppose that the first three cells in this ladder are all 2-cells which do not lie in $A$,
	for otherwise $d_A(p'',q'') \leq 3l$.
	Denote these three cells by $R_{xy}$, $R_1$ and $R_2$.
	The boundary of $R_1$ consists of four segments of lengths $k_1, \ldots k_4$ meeting
	$R_{xy}$, $\gam_{1x}$, $R_2$ and $\gam_{1y}$, respectively 
	(Figure~\ref{fig-qiembed2}).
	
	Observe that $k_1+k_3 \leq (3d+\frac{3}{2}\eps)l$ by Theorem~\ref{thm-O-sc-rand}
	applied to $R_{xy} \cup R_1 \cup R_2$,
	and $k_2\leq l/2$.  So assuming $l$ is large enough depending on $d<\frac{1}{8}$,
	we have $k_4 \geq 2\eta+l/8$.  Likewise, $k_2 \geq 2\eta+l/8$.
\begin{figure}
	\centering
	\def\svgwidth{0.75\columnwidth}
	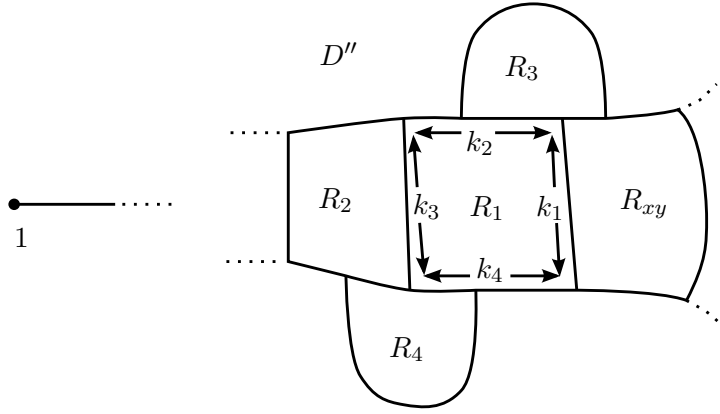
	\caption{Finding faces in $A$}\label{fig-qiembed2}
\end{figure}
	
	By Proposition~\ref{prop-startrel} and Lemma~\ref{lem-top-embed}, we have that every geodesic path
	from a point in $A \subset X$ to $1$ 
	alternates between paths in boundaries of $2$-cells of $A$, and segments $[u,u']$ 
	of length $\eta$ so
	that no geodesic $[u,1]$ starts with $\eta+l/8$ of some relation.
	This means that all of the length $k_2$, with the exception of $2\eta+l/8$, must lie in the
	boundary of some 2-cell $R_3 \subset A$.
	Likewise, all but $2\eta+l/8$ of the length $k_4$ lies in the boundary of some 2-cell $R_4 \subset A$.
	
	Consider the diagram $D''$ formed of $R_{xy} \cup R_1 \cup \cdots \cup R_4$.  
	This is reduced as $R_3$ and $R_4$ lie in $A$, while the other 2-cells do not.
	All but $4\eta+l/4$ of $\partial R_1$ lies in the interior of $D''$, so
	$|\partial D| \leq 5l-2(l-(4\eta+l/4))=(7/2)l+8\eta$.
	On the other hand, by Theorem~\ref{thm-O-sc-rand}, $|\partial D| \geq (1-2d-\eps)5l > (15/4)l$.
	This gives a contradiction for sufficiently large $l$, so $d_A(p'',q'') \leq 3l$.
	
	Note that $\gam_{1x}$ and $\gam_{1y}$ are in $A^{(1)}$ (Lemma~\ref{lem-top-embed}).
	Combining with \eqref{eq-dGam-bound}, we conclude the proof.
	\begin{align*}
		d_A(x,y) & \leq d_A(x,p'')+d_A(p'',q'')+d_A(q'',y) \\
			& \leq d_A(x,p'')+3l+d_A(q'',y) \\
			& = d_X(x,p'')+3l+d_X(q'',y) 
			\leq d_X(x,y) + 8l.\qedhere
	\end{align*}
\end{proof}

\subsection{Conformal dimension bound}\label{ssec-density-cdim}

We have built a quasi-isometrically embedded combinatorial round tree $A$ in $X$,
with vertical branching $|T|$, where
\[
	\log |T|=(K-\eta-1)\log(2m-1) = (\lfloor M^*/3 \rfloor-\eta-1)\log(2m-1).
\]
We choose $M^* = \lceil \frac{4}{5} dl \rceil$, and observe
that the presentation for $G$ contains every reduced word of length $M^*$
by~\cite[Proposition 2.7]{Mac-12-random-cdim}.
So for large enough $l$, $\log |T| \geq \frac{1}{4}dl \log(2m-1)$.

The horizontal branching of $A$ is most $l/(M^*/6)+2 \leq 8/d+2 \leq 24/d$,
so by Theorem~\ref{thm-roundtree-cdim} we have
\[
	\Cdim(\bdry G) \geq 1+\frac{dl \log(2m-1)}{4\log(24/d)}.
\]
The proof of Theorem~\ref{thm-main-density} is complete.\qed

\bibliographystyle{alpha}
\bibliography{biblio}

\end{document}